\documentclass[11pt, letterpaper]{article}
\usepackage[american]{babel}
\usepackage[T1]{fontenc}
\usepackage{amsmath, amssymb, amsfonts}
\usepackage{amsthm}
\usepackage{mathtools}
\usepackage{mathrsfs}
\usepackage{enumerate}
\usepackage{graphicx}
\usepackage{color}
\usepackage{enumitem}
\usepackage{xfrac}
\usepackage{dsfont}

\usepackage{epstopdf}

\usepackage{thumbpdf}
\usepackage{hyperref}
\usepackage[margin = 1in]{geometry}

\usepackage{lipsum}

\theoremstyle{plain}
\newtheorem{theorem}{Theorem}[section]
\newtheorem{lemma}[theorem]{Lemma}
\newtheorem{remark}[theorem]{Remark}
\newtheorem{corollary}[theorem]{Corollary}
\newtheorem{definition}[theorem]{Definition}
\newtheorem{assumption}[theorem]{Assumption}

\renewcommand{\d}{\mathrm{d}}

\numberwithin{equation}{section}

\def\MATLAB{}

\allowdisplaybreaks

\begin{document}

\title{On a Kelvin-Voigt Viscoelastic Wave Equation with Strong Delay}

\author{%
Andrii Anikushyn\thanks{Department of Computer Sciences and Cybernetics, Taras Shevcheno National University of Kyiv, Ukraine \texttt{anik\_andrii@univ.kiev.ua}} \and
Anna Demchenko\thanks{Department of Applied Mathematics \& Computer Science, Masaryk University, Brno, Czech Republic \hfill \texttt{ademchenko17@gmail.com}} \and
Michael Pokojovy\thanks{Department of Mathematical Sciences, The University of Texas at El Paso, TX, USA \hfill \texttt{mpokojovy@utep.edu}}
}

\date{\today}

\maketitle

\begin{abstract}
    An initial-boundary value problem for a viscoelastic wave equation subject to a strong time-localized delay in a Kelvin \& Voigt-type material law is considered.
    Transforming the equation to an abstract Cauchy problem on the extended phase space, 
    a global well-posedness theory is established using the operator semigroup theory both in Sobolev-valued $C^{0}$- and $\mathrm{BV}$-spaces.
    Under appropriate assumptions on the coefficients, a global exponential decay rate is obtained 
    and the stability region in the parameter space is further explored using the Lyapunov's indirect method.
    The singular limit $\tau \to 0$ is further studied with the aid of the energy method.
    Finally, a numerical example from a real-world application in biomechanics is presented.
\end{abstract}

\begin{center}
\begin{tabular}{p{1.0in}p{5.0in}}
    \textbf{Key words:} &  wave equation, Kelvin-Voigt damping,  time-localized delay, well-posedness, exponential stability, singular limit  \\
    \textbf{MSC (2010):} &
    35Q74,  
    74D05,  
    74H20,  
    74H25,  
    74H30,  
    74H55,  
    39A06   
\end{tabular}
\end{center}

\section{Introduction}\label{SECTION_INTRODUCTION}


Consider a viscoelastic body occupying in its reference configuration, in which renders the body is free of any stresses, a bounded Lipschitz domain $\Omega \subset \mathbb{R}^{d}$, $d \in \mathbb{N}$.
With $U \colon [0, \infty) \times \bar{\Omega} \to \mathbb{R}^{d}$ denoting the deflection field
and $\sigma \colon [0, \infty) \times \bar{\Omega} \to \mathbb{R}^{d \times d}$ stand for the Cauchy stress tensor,
the balance of momentum equation reads as
\begin{equation}
    \label{EQUATION_BALANCE_OF_MOMENTUM}
    \rho \partial_{tt} U(t, x) + \operatorname{div} \sigma(t, x) = F(t, x) \quad \text{ for } \quad t > 0, \quad x \in \Omega,
\end{equation}
where $\rho > 0$ denotes the constant material density and $F \colon (0, \infty) \times \Omega \to \mathbb{R}^{d}$ stands for the specific volumetric external force.
Introducing the infinitesimal (Lagrangian) strain tensor
\begin{equation}
    \label{EQUATION_STRAIN_TENSOR}
    \varepsilon(t, x) = \frac{1}{2} \Big(\nabla U(t, x) + \big(\nabla U(t, x)\big)^{T}\Big) \quad \text{ with } \quad
    \nabla = (\partial_{x_{1}}, \dots, \partial_{x_{d}})^{T}
\end{equation}
a Hook-type homogeneous isotropic Kelvin \& Voigt viscoelastic material law reads as
\begin{equation}
    \label{EQUATION_HOOK_LAW_KELVIN_VOIGT}
    \sigma(t, x) = \big(2 \mu \varepsilon(t, x) + 2 \tilde{\mu} \partial_{t} \varepsilon(t, x)\big) + 
    \Big(\lambda \operatorname{tr}\big(\varepsilon(t, x)\big) + \tilde{\lambda} \operatorname{tr}\big(\partial_{t} \varepsilon(t, x)\big)\Big) \mathrm{I}_{d \times d}
\end{equation}
with Lam\'{e} parameters $\lambda$ and $\mu$ 
and their anelastic counterparts (`viscosities') $\tilde{\lambda}$ and $\tilde{\mu}$ (cf. \cite[Section 2]{CaPoGe2004}).
Combining Equations (\ref{EQUATION_BALANCE_OF_MOMENTUM})--(\ref{EQUATION_HOOK_LAW_KELVIN_VOIGT}),
we arrive at the following system of linear viscoelasticity:
\begin{equation}
    \label{EQUATION_SYSTEM_KELVIN_VOIGT}
    \rho \partial_{tt} U(t, x) - \mu \triangle U(t, x) - (\mu + \lambda) \operatorname{div} U(t, x)
    - \tilde{\mu} \triangle \partial_{t} U(t, x) - (\tilde{\mu} + \tilde{\lambda}) \operatorname{div} \partial_{t} U(t, x) = \rho F(t, x).
\end{equation}
We refer the reader to \cite{KaLa2009} for a related model describing propagation of sound in fluids.

In the present paper, we propose to replace the instanteneous Kelvin \& Voigt law with the following time-localized delay constitutive relation
\begin{align}
    \label{EQUATION_HOOK_LAW_KELVIN_VOIGT_DELAY}
    \sigma(t, x) &= \Big(2 \mu \varepsilon(t, x) + 2 \tilde{\mu} \partial_{t} \varepsilon(t, x)\big) + 
    \Big(\lambda \operatorname{tr}\big(\varepsilon(t, x)\big) \mathrm{I}_{d \times d} + \tilde{\lambda} \operatorname{tr}\big(\partial_{t} \varepsilon(t, x)\big) \mathrm{I}_{d \times d}\Big) \\
    \notag
    &+ \Big(2 \mu_{\tau} \varepsilon(t - \tau, x) + 2\tilde{\mu}_{\tau} \partial_{t} \varepsilon(t - \tau, x)\big) + 
    \Big(\lambda_{\tau} \operatorname{tr}\big(\varepsilon(t - \tau, x)\big) + \tilde{\lambda}_{\tau} \operatorname{tr}\big(\partial_{t} \varepsilon(t - \tau, x)\big)\Big) \mathrm{I}_{d \times d}
\end{align}
with constants $\mu_{\tau}, \lambda_{\tau}, \tilde{\mu}_{\tau}, \tilde{\mu}_{\tau}$ such that $\mu + \lambda > 0$ and a positive delay time $\tau > 0$.
The new material law (\ref{EQUATION_HOOK_LAW_KELVIN_VOIGT_DELAY}) is more comprehensive as it allows for a trade-off between the standard instanteneous law (\ref{EQUATION_SYSTEM_KELVIN_VOIGT})
and the purely retarded relation
\begin{equation}
    \notag
    \sigma(t, x) = \Big(2 \mu_{\tau} \varepsilon(t - \tau, x) + 2 \tilde{\mu}_{\tau} \partial_{t} \varepsilon(t - \tau, x)\big) + 
    \Big(\lambda_{\tau} \operatorname{tr}\big(\varepsilon(t - \tau, x)\big) + \tilde{\lambda}_{\tau} \operatorname{tr}\big(\partial_{t} \varepsilon(t - \tau, x)\big)\Big) \mathrm{I}_{d \times d}.
\end{equation}
In contrast to the instanteneous law (\ref{EQUATION_HOOK_LAW_KELVIN_VOIGT}) rendering the system (\ref{EQUATION_SYSTEM_KELVIN_VOIGT}) parabolic and exponentially stable (thus, `overdamping' the system),
the mixed delay constitutive equation (\ref{EQUATION_HOOK_LAW_KELVIN_VOIGT_DELAY}) can adequately describe both stable and unstable dynamics 
as $t \to \infty$ depending on the choice of constants and parameters (cf. \cite{KhuPoRa2015}).
In addition, the delay material law (\ref{EQUATION_HOOK_LAW_KELVIN_VOIGT_DELAY}) is also advantageous on finite time horizons
as it provides for more flexibility due to the presence of additional degrees of freedom from `delayed viscosity' parameters.
In particular, the model is capable of modeling non-monotonic dynamics (cf. Section \ref{SECTION_NUMERICAL_EXAMPLE}).
For a more detailed review of the positive impact of delay mechanisms in physics and engineering we refer to \cite{AbDoBeBy1993, JaOl1998, KwoLeeKi1990, SuBie1980}.
At the same time, a great degree of carefulness is required when studying and designing delay material laws and feedbacks
since delay has the potential to destabilize the system or even destroy well-posedness \cite{DaLaPo1986, Da1988, KhuPoRa2015}.

For the sake of simplicity, in this paper, we assume the material is incompressible, i.e., $\operatorname{div} U \equiv 0$, and let $F \equiv 0$.
Further, we assume there exists a potential function $y \colon [0, \infty) \times \bar{\Omega} \to \mathbb{R}$ such that $U(t, x) = \nabla y(t, x)$.
Rewriting the Laplacian via the $\operatorname{curl}$-operator
$\triangle U = \operatorname{curl} \operatorname{curl} U - \nabla \operatorname{div} U$, recalling $\operatorname{curl} \nabla U \equiv 0$ and renaming the constants accordingly, 
Equations (\ref{EQUATION_BALANCE_OF_MOMENTUM}), (\ref{EQUATION_STRAIN_TENSOR}), (\ref{EQUATION_HOOK_LAW_KELVIN_VOIGT_DELAY}) can be reduced to
\begin{equation}
    \label{EQUATION_KELVIN_VOIGT_EQUATION_ORIGINAL}
    \partial_{tt} y(t, x) - c_1 \triangle y(t, x) - c_2 \triangle y(t-\tau, x)
    -d_1 \partial_t \triangle y(t,x)-d_2 \partial_t \triangle y(t-\tau,x) = f(t, x)
\end{equation}
with $f(t, x) = \rho \operatorname{div} F(t, x)$.
Alternatively, assuming $\operatorname{div} U \equiv 0$, the equations for $U_{1}, U_{2}, \dots, U_{d}$ decouple
so that we can obtain Equation (\ref{EQUATION_KELVIN_VOIGT_EQUATION_ORIGINAL}) for, say, $y \equiv U_{2}$ (cf. Section \ref{SECTION_NUMERICAL_EXAMPLE}).
We refer the reader to \cite{ChaCro1986} for a more rigorous reduction procedure.

To close Equation (\ref{EQUATION_KELVIN_VOIGT_EQUATION_ORIGINAL}), boundary and initial conditions are required.
We assume $\Gamma$ is subdivided into two relatively open parts $\Gamma_0$ and $\Gamma_1$, i.e., $\Gamma = \bar{\Gamma}_1 \cup \bar{\Gamma}_0$, 
with $\bar{\Gamma}_1 \cap \bar{\Gamma}_0 = \emptyset$ and $\Gamma_{0} \neq \emptyset$.
From the physical point of view, two types of boundary conditions are natural for $y$: either prescribing the deflection $y$ or or the surface stress
\begin{equation}
    \notag
    c_1 \frac{\partial y}{\partial \nu}(t, x) - c_2 \frac{\partial y}{\partial \nu} y(t-\tau, x)
    -d_1 \partial_t \frac{\partial y}{\partial \nu} y(t,x) - d_2 \partial_t \frac{\partial y}{\partial \nu}  y(t - \tau,x) = 0.
\end{equation}
The latter is a neutral differential equation (viz. \cite[Chapter 9]{HaVeLu1993}) and, therefore, uniquely solvable for $\frac{\partial y}{\partial \nu}(t, \cdot)$.
Thus, prescribing the surface stress is equivalent with providing a boundary condition for $\frac{\partial y}{\partial \nu}(t, \cdot)$.
Due to obvious analytic benefits of this approach (operator domain is constant, etc.), we will pursue this approach.

Imposing mixed homogeneous Dirichlet--Neumann boundary conditions on $y$ and usual initial conditions,
we obtain a system of partial delay differential equations
\begin{align}
    \notag
    \partial_{tt} y(t, x) - c_1 \triangle y(t, x) - c_2 \triangle y(t-\tau, x) &\phantom{=} \\
    \label{INTRODUCTION_EQUATION}
    -d_1 \partial_t \triangle y(t,x)-d_2 \partial_t \triangle y(t-\tau,x) &= 0 \text{ for } t > 0, \; x \in \Omega, \\
    \label{INTRODUCTION_BC_2}
    y(t,x) = 0 \text{ for } t > 0, \; x \in \Gamma_0, \quad \frac{\partial y(t,x)}{\partial \nu} &= 0 \text{ for }  t > 0, \; x \in \Gamma_1, \\
    \label{INTRODUCTION_IC_2}
    y(0+, x) = y^0, \quad \partial_t y(0+, x) = y^{1} \text{ for } x \in \Omega, \quad y(t, x) &= \varphi(t, x) \text{ for } (t, x) \in  [-\tau,0] \times \Omega,
\end{align}
where $\nu \colon \Gamma \to \mathbb{R}^{3}$ stands for the outer unit normal vector to the boundary $\Gamma$ 
and $\frac{\partial u}{\partial \nu}$ is the normal derivative, $\tau > 0$ is a delay time, and $c_1,c_2,d_1,d_2$ are positive real numbers.


Delay effects widely arise in a variety of areas of science and engineering, in particular physics, biology, chemistry, ecology, material engineering, mechanism design, etc. 
Such phenomena are commonly modeled by ordinary and partial delay differential equations.
To study these problems mathematically, a new area of analysis and control of delay differential equations has emerged.
Over the past decades, numerous systems similar to Equations (\ref{INTRODUCTION_EQUATION})-(\ref{INTRODUCTION_IC_2}) have massively drawn attention in various communities.
In the following, we give a brief review of most relevant publications.

Ammari {\it et al.} \cite{AmNiPi1995} considered an abstract equation reminiscent of (\ref{INTRODUCTION_EQUATION})--(\ref{INTRODUCTION_IC_2}) (for $c_{1} = d_{2} = 0$):
\begin{align}
    \label{EQUATION_AMMARI_NICAISE_PIGNOTTI_EQ1}
    \partial_{tt} u(t) + a \mathcal{B} \mathcal{B}^{\ast} \partial_{t} u(t) + \mathcal{B} \mathcal{B}^{\ast} u(t - \tau) &= 0 \quad \text{ for } t > 0, \\
    \label{EQUATION_AMMARI_NICAISE_PIGNOTTI_EQ2}
    u(0) = u^{0}, \quad \partial_{t} u(0) &= u^{1} \quad \text{ for } t > 0, \\
    \label{EQUATION_AMMARI_NICAISE_PIGNOTTI_EQ3}
    \mathcal{B}^{\ast} u(t) &= f_{0}(t) \quad \text{ for } t \in [-\tau, 0]
\end{align}
for an appropriate class of operators $\mathcal{B}$. Here, $a > 0$ and $\mathcal{B} \mathcal{B}^{\ast}$ models general `elliptic' operators such as the Dirichlet Laplacian.
Using the semigroup theory, well-posedness of Equations (\ref{EQUATION_AMMARI_NICAISE_PIGNOTTI_EQ1})--(\ref{EQUATION_AMMARI_NICAISE_PIGNOTTI_EQ3})
was shown on the extended phase space. The spectrum of underlying evolution operator was investigated and an exponetial stability result was proved.
Despite its superficial similarity with Equations (\ref{INTRODUCTION_EQUATION})--(\ref{INTRODUCTION_IC_2}),
the delay in our Equation (\ref{INTRODUCTION_EQUATION}) is stronger due to the presence of $\triangle \partial_{t} y(t - \tau, \cdot)$-term.

Ammari {\it et al.}~\cite{AmNiPi2010} studied the problem of boundary stabilization of a wave equation with a constant delay in the velocity field over an open bounded domain $\Omega \subset \mathbb{R}^d$ 
\begin{align}
    \partial_{tt} u(x,t) - \triangle u(x,t) + a\partial_{t}u(x,t-\tau) &= 0, \quad x \in \Omega, \quad t>0, \\
    u(x,t) &= 0, \quad x \in \Gamma_0, \quad t > 0, \\
    \frac{\partial u}{\partial \nu} (x, t) &= -k \partial_{t} u(x,t), \quad x\in \Gamma_1, \quad t>0
\end{align}
subject to usual initial conditions for $u(\cdot, 0)$, $\partial_{t} u(\cdot ,0)$ and $u|_{\Omega \times (-\tau, 0)}$.
Here, $\partial \Omega = \bar{\Gamma}_0 \cup \bar{\Gamma}_1$ is the boundary of $\Omega$ and $a, b > 0$ are two constants. 
Using multiplier techniques and obtaining a uniform decay estimate for a suitable Lyapunov functional, an exponential stability result in an appropriate norm was shown.

Benaissa \& Messaoudi \cite{BeMe2013} investigated decay properties of solutions to a wave equation with a time-varying delay in the velocity field
\begin{align}
    \notag
    \partial_{tt}(x,t) - \triangle u(x,t) + \mu_1 \sigma(t) \partial_{t} u(x,t) & \\
    \label{INTRO_EQ1}
    + \mu_2(t)\sigma(t) u_t\big(x,t-\tau(t)\big) + \theta(t) h\big(\nabla_x u(x,t)\big) &= 0, \quad (x, t) \in \Omega \times (0, \infty), \\
    \label{INTRO_EQ2}
    u(x,t) &= 0, \quad (x, t) \in \Gamma \times (0, \infty)
\end{align}
together with initial conditions for $u(\cdot, 0)$, $\partial_{t} u(\cdot, 0)$ and $u|_{\Omega \times (-\tau(0), 0)}$.
Adopting some of Guesmia's ideas \cite{Gu2003_1, Gu2003_2} and using multiplier techniques, 
an energy decay estimate was proved for solutions to Equations (\ref{INTRO_EQ1})--(\ref{INTRO_EQ2}) both for linear and nonlinear $h(\cdot)$'s.

Nicaise \& Pignotti \cite{NiPi2006} examined a wave equation with a delay term in the boundary feedback
\begin{align}
    \label{INTRO_EQ3}
    \partial_{tt}(x,t)-\triangle u(x,t) &= 0, \quad  x\in\Omega, \quad t>0, \\
    u(x,t) &= 0, \quad x\in \Gamma_D, \quad t>0, \\
    \label{INTRO_EQ4}
    \frac{\partial u}{\partial \nu} (x,t) + \mu_1 \partial_{t} u(x,t) + \mu_2 \partial_{t} u(x,t-\tau) &= 0, \quad x\in \Gamma_N, \quad t>0
\end{align}
subject to appropriate initial conditions -- and a second model with a delayed feedback distributed over $\Omega$ subject to homogeneous boundary conditions.
Here, the boundary $\Gamma = \partial \Omega$ is decomposed as $\Gamma = \bar{\Gamma}_D \cup \bar{\Gamma}_N$ with $\Gamma_D \cap \Gamma_N = \emptyset$ and $\Gamma_D \neq \emptyset$.
Assuming $\mu_2 < \mu_1$, exploiting Carleman estimates established by Lasiecka {\it et al.} \cite{LaTrYa1999} and adopting compactness-uniqueness arguments, 
a stabilization result for (\ref{INTRO_EQ3})--(\ref{INTRO_EQ4}) was obtained in general space dimension by proving a suitable observability estimate.
Moreover, for $\mu_2 \ge \mu_1$, the existence of a sequence of delays destabilizing the system was shown. 
Further, under appropriate assumptions on $a \in L^\infty(\Omega)$, on the strength of a suitable observability estimate,
the same result was also obtained for the model with the delay in the interior.
These results were further extended by the same authors \cite{NiPi2011} to the case of time-varying delays.

Nicaise \& Pignotti \cite{NiPi2008} also examined a wave equation subject to Dirichlet boundary conditions 
on one part of the boundary and a dissipative boundary condtion with distributed delay on the remaining part of the boundary:
\begin{align}
    \label{INTRO_EQ13}
    \partial_{tt} u(x,t)-\triangle u(x,t) &= 0, \quad x \in \Omega, \quad t>0, \\
    u(x,t) &= 0, \quad x \in \Gamma_D, \quad t>0, \\
    \label{INTRO_EQ14}
    \frac{\partial u}{\partial \nu} (x,t) + \mu_0 u_t(x,t) +  \int_{\tau_1}^{\tau_2}\mu(s) u_t(x,t-s) \mathrm{d}s &= 0, \quad x\in \Gamma_N, \quad t>0
\end{align}
subject to suitable initial conditions along with a counterpart of the former system with an internally distributed delayed feedback.
Utilizing the semigroup theory, well-posedness and exponential stability under appropriate conditions were established for both systems.
Furthermore, an instability counter-example for the case of internal delayed feedback was constructed in case
the condition $\mu_0 > \|a\|_{\infty} \int_{\tau_1}^{\tau_2} \mu(s) \mathrm{d}s$ is violated.

Motivated by \cite{NiPi2006}, Kirane \& Said-Houari \cite{KiSa2011} considered an initial-boundary value problem for a viscoelastic wave equation with a delay term in the interior feedback
\begin{align}
    \notag
    \partial_{tt} u(x,t) - \triangle u(x,t) + \int_0^t g(t-s)\triangle u(x,s) \mathrm{d}s & \\
    \label{INTRO_EQ7}
    + \mu_1 \partial_{t} u(x,t) + \mu_2(t) \partial_{t} u(x, t-\tau) &= 0, \quad  x \in \Omega, \quad t > 0, \\
    \label{INTRO_EQ8}
    u(x,t) &= 0, \quad x \in \partial \Omega, \quad t > 0
\end{align}
subject to initial conditions on $u(\cdot, 0)$, $\partial_{t} u(\cdot ,0)$ and $u|_{\Omega \times (-\tau, 0)}$.
Equations (\ref{INTRO_EQ7})--(\ref{INTRO_EQ8}) arise as a linear model for propagation of a viscoelastic wave in compressible fluids.
Under the assumption $0 < \mu_2 \le \mu_1$, using energy estimates and a Faedo \& Galerkin approximation, the system was showed to be well-posed. 
Further, under the same assumptions, the authors proved a general decay rate for the total energy of Equations (\ref{INTRO_EQ7})--(\ref{INTRO_EQ8}). 
Due to the presence of viscoelastic damping, the solution was shown to be asymptotically stable even if $\mu_1 = \mu_2$ -- in contrast to \cite{NiPi2006} where no viscoelastic damping was available.

Dai \& Yang \cite{DaYa2014} improved on the results of \cite{NiPi2006} for Equations (\ref{INTRO_EQ7})--(\ref{INTRO_EQ8}). 
Controlling the delay term by the derivative of the energy instead of the damping term in the {\it a priori} estimate,
well-posedness  under suitable assumptions on the kernel $g(\cdot)$ without assuming the positivity of $\mu_1, \mu_2$ and $\mu_2 \le \mu_1$ was shown.
Introducing a new Lyapunov functional, an energy decay result for the problem (\ref{INTRO_EQ7})--(\ref{INTRO_EQ8}) for the case $\mu_1 = 0$ was proved as well. 

Benaissa \& Benguessoum \cite{BeBe2014} investigated the existence and decay properties of solutions to the viscoelastic wave equation 
with a delay term in the nonlinear internal feedback of the form
\begin{align}
    \notag
    \partial_{tt} u(x,t) - \triangle u(x,t) + \int_0^t h(t-s) \triangle u(x,s) \mathrm{d}s & \\
    + \mu_1 g_1\big(\partial_{t} u(x,t)\big) + \mu_2 g_2\big(\partial_{t} u(x,t-\tau)\big) &= 0, \quad (x,t) \in \Omega \times (0, \infty), \\
    u(x,t) &= 0, \quad (x,t) \in \Gamma \times (0, \infty)
\end{align}
along with initial conditions for $u(\cdot, 0)$, $\partial_{t} u(\cdot ,0)$ and $u|_{\Omega \times (-\tau, 0)}$ in a bounded domain $\Omega \subset \mathbb{R}^{d}$. 
Here, $h(\cdot)$ is a positive non-increasing function defined on $\mathbb{R}^+$, $g_1(\cdot)$ and $g_2(\cdot)$ are two kernel functions, $\tau > 0$ is a delay time,
and $\mu_1, \mu_2$ are positive real numbers.
Under some additional assumptions, using the energy method combined with the Faedo \& Galerkin scheme,
global existence result was established and the long-time behavior of solutions was examined using a perturbed energy method.

Liu \cite{Li2013} investigated a viscoelastic equation with time-varying internal delay feedback:
\begin{align}
    \notag
    \partial_{tt} u(x,t) - \triangle u(x,t) + \int_0^t g(t-s)\triangle u(x,s) \mathrm{d}s & \\
    \label{INTRO_EQ11}
    + a_0  \partial_{t} u(x,t) + a_1 \partial_{t} u\big(x, t-\tau(t)\big) &= 0, \quad x \in \Omega, \quad t > 0, \\
    u(x,t) &= 0, \quad x \in \partial \Omega, \quad t \ge 0
\end{align}
subject to initial conditions on $u(\cdot, 0)$, $\partial_{t} u(\cdot ,0)$ and $u|_{\Omega \times (-\tau, 0)}$.
The boundary $\partial \Omega$ was assumed of $C^2$-class, the functions $\alpha(\cdot), g(\cdot)$ positive and non-increasing, $a_0 > 0$ and $\tau(t) > 0$. 
Under suitable assumptions, a general decay rate for the energy was proved depending both on $\alpha$ and the behavior of $g(\cdot)$.
Summarizing, this work extended the results of \cite{KiSa2011, WLiu2013, NiPi2006} to time-varying delays without positivy assumption on $a_1$.

Xu {\it et al.} considered a 1D wave system with a homogeneous Dirichlet boundary condition at $x = 0$ a delayed Neumann boundary feedback at $x = 1$:
\begin{align}
    \label{INTRO_EQ17}
     \partial_{tt} u(x,t) - \partial_{xx} u(x,t) &= 0, \quad x \in (0, 1), \quad t > 0, \\
    u(0, t) &= 0, \quad t > 0, \\
    \label{INTRO_EQ18}
    \partial_{x} u(1,t) + k \mu  \partial_{t} u(1, t) + k(1 - \mu)  \partial_{t} u(1, t - \tau) &= 0, \quad t > 0
\end{align}
complemented by appropriate initial conditions on $u(\cdot, 0)$, $\partial_{t} u(\cdot ,0)$ and $u|_{\{1\} \times (-\tau, 0)}$.
The closed-loop system (\ref{INTRO_EQ17})--(\ref{INTRO_EQ18}) was shown to generates a $C_0$-semigroup of linear bounded operators and the stability of the system was investigated.
In particular, it was proved,  given $\mu > \frac12$, for any $\tau > 0$, it is possible to choose $k > 0$ so that the closed-loop system is exponentially stable. 
When $\mu=\frac{1}{2}$, it was further shown that for rational $\tau \in (0, 1)$ the system is unstable, 
whereas for irrational $\tau \in (0, 1)$ the system is asymptotically stable. When $\mu < \frac{1}{2}$, the system is always unstable.  

Despite its structural similarity with the `pure' wave equation, the instanteneous part of Equations (\ref{INTRODUCTION_EQUATION})--(\ref{INTRODUCTION_IC_2})
rather exhibits a parabolic nature manifested in the analyticity of the semigroup generated by the respective driving operator (cf. \cite[Theorem 1.1 with $\beta = 0$, $\sigma = 1/2$]{Ma1983})
Numerous publications both on the heat equation with delay and abstract parabolic systems are available in the literature.
For space considerations, we omit a review and rather refer the reader to Khusainov {\it et al.} \cite{KhuPoRa2015} for a detailed summary of recent literature on this topic.
In contrast to \cite{KhuPoRa2015} and many references therein, 
we do not exploit the analyticity of the instanteneous part of Equations (\ref{INTRODUCTION_EQUATION})--(\ref{INTRODUCTION_IC_2})
(except for the regularity study in Section \ref{SECTION_SOLUTION_REGULARITY}) but rather employ general semigroup arguments.
This way, Equations (\ref{INTRODUCTION_EQUATION})--(\ref{INTRODUCTION_IC_2}) are treated by the same technique as those applied to hyperbolic systems.
Similar approach was adopted by Ammari {\it et al.} \cite{AmNiPi1995} for an abstract system reminiscent of Equations (\ref{INTRODUCTION_EQUATION})--(\ref{INTRODUCTION_IC_2}).


The rest of the paper is organized as follows. In Section \ref{SECTION_WP}, we establish well-posedness for Equations (\ref{INTRODUCTION_EQUATION})--(\ref{INTRODUCTION_IC_2})
by transforming the equation to extended phase space and using operator semigroup theory in $L^{2}$-Sobolev-valued $C^{0}$- and $\mathrm{BV}$-spaces.
Further, regularity of solutions is investigated.
In Section \ref{SECTION_LONG_TIME_BEHAVIOR}, exponential decay rate of Equations (\ref{INTRODUCTION_EQUATION})--(\ref{INTRODUCTION_IC_2})
is established using Lyapunov's method and the stability region in the parameter space is further investigated using sufficient conditions obtained with he Lyapunov's method.
Section \ref{SECTION_TAU_TO_ZERO} is dedicated to studying the singular limit of Equations (\ref{INTRODUCTION_EQUATION})--(\ref{INTRODUCTION_IC_2}) as $\tau \to 0$.
Finally, Section \ref{SECTION_NUMERICAL_EXAMPLE} presents a real-world numerical example illustrating our model.


\section{Solution Theory} \label{SECTION_WP}
In the following, assume $c_{1}, d_{1} > 0$ and $c_{2}, d_{2} \neq 0$.
Introducing the `history variable'
\begin{equation}
    \label{WELL_POS_FUNCTION_Z_DEFINITION}
    z(s, t, x) = y(t-\tau s, x), \quad s \in (0,1), \quad t > 0,
\end{equation}
we define the extended phase space
\begin{equation}
    \notag
    \mathscr{H} =  H_{\Gamma_0}^1(\Omega) \times L^2(\Omega) \times L_{\tau}^2\big(0, 1; H_{\Gamma_0}^1(\Omega)\big) \times L_{\tau}^2\big(0, 1; H_{\Gamma_0}^1(\Omega)\big)
\end{equation}
equipped with the inner product
\begin{equation}
    \label{WELL_POS_SPACE_H_INNER_PRODUCT}
    \langle U, V \rangle_{\mathscr{H}} = \int_G \Big(c_1 \nabla u_1 \cdot \nabla v_1 + u_2 v_2 + \tau d_1 \int_0^1 \xi e^{cs} \nabla u_3 \cdot \nabla v_3\ \mathrm{d}s
    + \tau |d_2| \int_0^1 \xi e^{cs} \nabla u_4 \cdot \nabla v_4\ \mathrm{d} s \Big) \mathrm{d}x
\end{equation}
for $U=(u_1,u_2,u_3,u_4)^{T}, V=(v_1,v_2,v_3,v_4)^{T} \in \mathscr{H}$,
where $H_{\Gamma_0}^1(\Omega) = \big\{u \in H^1(\Omega) \,|\, u|_{\Gamma_0} = 0\big\}$ and 
\begin{align}
    \label{WELL_POS_CONDITIONS_ON_XI}
    \xi = \frac{d_1}{3|d_2|}, \quad c > \ln\left(\max \Big\{ \frac{9c_2^2|d_2|}{d_2^3}, \frac{9d_2^2}{d_1^2}\Big\}\right).
\end{align} 
It is easy to verify that the topology induced by the inner product is equivalent with the standard product topology on $\mathscr{H}$.

Now, introducing the extended state variable $(y, \partial_{t} y, z)$,
Equations (\ref{INTRODUCTION_EQUATION})--(\ref{INTRODUCTION_IC_2}) can equivalently be written as
\begin{align}
    \notag
    \partial_{tt} y(t,x)-c_1 \triangle y(t,x)-c_2 \triangle z(t, 1, x) & \\
    \label{WELL_POS_INITIAL_SYSTEM_EQUATION}
    -d_1 \partial_t \triangle y(t,x)-d_2 \partial_t \triangle z(t, 1, x) &= 0 \text{ for } (t, x) \in (0, \infty) \times \Omega, \\
    \label{WELL_POS_INITIAL_SYSTEM_Z_1}
    \partial_t z(t, s, x) + \tfrac{1}{\tau} \partial_s z(t, s, x) &=0 \text{ for } (t, s, x) \in (0, \infty) \times (0,1) \times \Omega, \\
    \label{WELL_POS_INITIAL_SYSTEM_Z_2}
    \partial_{tt} z(t, s, x) + \tfrac{1}{\tau} \partial_{st} z(t, s, x) &= 0 \text{ for } (t, s, x) \in (0, \infty) \times (0,1) \times \Omega, \\
    \label{WELL_POS_INITIAL_SYSTEM_BC_2}
    y(t,x) = 0 \text{ for } x \in \Gamma_0, \quad \frac{\partial y(t, x)}{\partial \nu} &= 0 \text{ for } x \in \Gamma_1, \\
    \label{WELL_POS_INITIAL_SYSTEM_IC_1}
    y(0+, x)=y^0(x), \quad \partial_t y(0+,x) &= y^{1}(x) \text{ for } x \in \Omega, \\
    \label{WELL_POS_INITIAL_SYSTEM_IC_2}
    y(t, x) &= \varphi(t, x) \text{ for } (t, x) \in (-\tau, 0) \times \Omega
\end{align}
so that the evolution is taking place on $\mathscr{H}$.
Letting $v_1 := y, v_2 := \partial_t y, v_3 := z, v_4 := \partial_t z$, we rewrite Equations (\ref{WELL_POS_INITIAL_SYSTEM_EQUATION})-(\ref{WELL_POS_INITIAL_SYSTEM_IC_2}) as follows:
\begin{align}
    \notag
    \partial_{t} v_2 (t,x)-c_1 \triangle v_1(t,x)-c_2 \triangle v_3(t, 1, x)& \\
    \label{WELL_POS_SECOND_SYSTEM_EQUATION}
    - d_1 \partial_t \triangle v_1(t,x) - d_2  \triangle v_4(t, 1, x) &= 0 \text{ for } (t, x) \in (0, \infty) \times \Omega, \\
    \label{WELL_POS_SECOND_SYSTEM_Z_1}
    \partial_{t} v_3(t, s, x) + \frac{1}{\tau} \partial_s v_3(t, s, x) &= 0 \text{ for } (t, s, x) \in (0, \infty) \times (0,1) \times \Omega, \\
    \label{WELL_POS_SECOND_SYSTEM_Z_2}
    \partial_{t} v_4(t, s, x) + \frac{1}{\tau} \partial_{s} v_4(t, s, x) &= 0 \text{ for } (t, s, x) \in (0, \infty) \times (0,1) \times \Omega, \\
    \label{WELL_POS_SECOND_SYSTEM_BC_2}
    v_1(t,x) = 0 \text{ for } x \in \Gamma_0, \quad \frac{\partial v_1(t,x)}{\partial \nu} &= 0 \text{ for } x \in \Gamma_1, \\
    \label{WELL_POS_SECOND_SYSTEM_IC_1}
    v_1(0+, x) = y^0(x), \quad v_2(0+, x) &= y^{1}(x) \text{ for } x \in \Omega, \\
    \label{WELL_POS_SECOND_SYSTEM_IC_2}
    v_1(t, x) = \varphi(t, x), \quad v_2(t, x) &= \partial_{t} \varphi(t, x) \text{ for } (t, x) \in (-\tau, 0) \times \Omega.
\end{align}

Consider the linear operator $\mathscr{A} \colon D(\mathscr{A}) \subset \mathscr{H} \to \mathscr{H}$ defined via
\begin{equation}
    \label{WELL_POS_OPERATOR_DEFINITION}
    \mathscr{A} V = 
    \begin{pmatrix} 
        v_2 \\ 
        c_1 \triangle v_1+c_2\triangle v_3|_{s=1}+d_1\triangle v_2+d_2\triangle v_4|_{s=1} \\ 
        -\tau^{-1} \partial_s v_3 \\ 
        -\tau^{-1} \partial_s v_4 
    \end{pmatrix}
    \quad \text{ for } \quad V = (v_1, v_2, v_3, v_4)^{T},
\end{equation}
where
\begin{equation}
    \label{WELL_POS_OPERATOR_DOMAIN}
    \begin{split}
        D(\mathscr{A}) = \big\{V = (v_1,v_2,v_3,v_4)^{T} \in \mathscr{H} \,\big|\, &\mathscr{A} V \in \mathscr{H} \text{ such that } v_3|_{s=0} = v_1, \; v_4|_{s=0} = v_2 \text{ and } \\
        &\text{Green's formula (\ref{EQUATION_GREEN_FORMULA}) holds true}\big\}
    \end{split}
\end{equation}
with the Green's formula
\begin{equation}
    \label{EQUATION_GREEN_FORMULA}
    \begin{split}
        \int_\Omega \big(c_1 \triangle v_1 &+ c_2 \triangle v_3|_{s=1} + d_1 \triangle v_2 + d_2 \triangle v_4|_{s=1}\big) \psi \, \mathrm{d}x \\
        &= -\int_\Omega \big(c_1 \nabla v_1 + c_2 \nabla v_3|_{s=1} + d_1 \nabla v_2 + d_2 \nabla v_4|_{s=1}\big) \cdot\nabla \psi \, \mathrm{d}x
    \end{split}
\end{equation}
for any $\psi \in H_{\Gamma_0}^1(\Omega)$.
Observing
\begin{align*}
    D(\mathscr{A}) = \big\{(v_1,v_2,v_3,v_4)^{T} \,\big|\, &v_1, v_2 \in H_{\Gamma_0}^1(\Omega), \;
    c_1 \triangle v_1 + c_2 \triangle v_3 |_{s=1} + d_1 \triangle v_2 + d_2 \triangle v_3 |_{s=1} \in L^2(\Omega), \\
    &v_3 \in H^1\big(0,1; H_{\Gamma_0}^1(\Omega)\big), \; v_4 \in H^1\big(0, 1; H_{\Gamma_0}^1(\Omega)\big)\big\},
\end{align*}
the domain $D(\mathscr{A})$ of $\mathscr{A}$ in Equation (\ref{WELL_POS_OPERATOR_DOMAIN}) can be expressed as
\begin{align}
    \notag
    D(\mathscr{A}) = \big\{V = (v_1,v_2,v_3,v_4)^{T} \,\big|\, &V \in H_{\Gamma_0}^1(\Omega) \times H_{\Gamma_0}^1(\Omega) \times H^1\big(0, 1; H_{\Gamma_0}^1(\Omega)\big) \times H^1\big(0, 1; H_{\Gamma_0}^1(\Omega)\big), \\
    \label{WELL_POS_OPERATOR_DOMAIN_REWRITTEN}
    &\text{s.t. } v_3|_{s=0} = v_1, \; v_4|_{s=0} = v_2 \text{ and Eq. (\ref{EQUATION_GREEN_FORMULA}) holds true}\Big\}.
\end{align}
Therefore, Equations (\ref{WELL_POS_SECOND_SYSTEM_EQUATION})--(\ref{WELL_POS_SECOND_SYSTEM_IC_2}) can be cast into the form of an abstract Cauchy problem on the extented phase space $\mathscr{H}$
\begin{equation}
    \label{WELL_POS_ABSTRACT_FORMULATION}
    \dot{V}(t) = \mathscr{A} V(t) \quad \text{ for } \quad t > 0, \quad V(0) = V^{0}
\end{equation}
with $V = (v_1, v_2, v_3, v_4)^{T}$ and $V^{0} := (y^{0}, y^{1}, \varphi^{0}, \varphi^{1})^{T}$,
where $\varphi^{0} := \varphi\big(-\frac{\cdot - \tau}{\tau}\big)$, $\varphi^{1} := (\partial_{t} \varphi)\big(-\frac{\cdot - \tau}{\tau}\big)$.

\subsection{Hadarmard Well-Posedness}
The following auxiliary lemma characterizes $\mathscr{A}$ as an $m$-dissipative operator on $\mathscr{H}$.
\begin{lemma}
    \label{WELL_POS_MAIN_LEMMA}
    The operator $\mathscr{A}$ defined in Equations (\ref{WELL_POS_OPERATOR_DEFINITION})--(\ref{WELL_POS_OPERATOR_DOMAIN}) 
    is an infinitesimal generator of a $C_0$-semigroup on $\mathscr{H}$.
\end{lemma}

\begin{proof}
    {\it Density: } Consider the set
    \begin{align*}
        \mathscr{D} = \Big\{V = (v_1,v_2,v_3,v_4)^{T} \,\big|\, &V \in C_{\Gamma_{0}}^{\infty}(\Omega) \times C_{\Gamma_{0}}^{\infty}(\Omega) \times C^{\infty}_{0}\big((0, 1), C_{\Gamma_{0}}^{\infty}(\Omega)\big) \times C^{\infty}_{0}\big((0, 1), C_{\Gamma_{0}}^{\infty}(\Omega)\big) \\
        &\text{such that } \frac{\partial v_{1}}{\partial \nu} v_{1}|_{\Gamma_{1}} = 0 \text{ and } \frac{\partial v_{1}}{\partial \nu} v_{2}|_{\Gamma_{1}} = 0\Big\}.
    \end{align*}
    By standard results, $\mathscr{D}$ is dense in $\mathscr{H}$.
    Thus, the set
    \begin{equation}
        \big\{(v_1, v_2, v_1 + v_3, v_2 + v_4)^{T} \,\big|\, (v_1, v_2, v_3, v_4) \in \mathscr{D}\big\} \subset D(\mathscr{A}) \notag
    \end{equation}
    is also dense in $\mathscr{H}$ implying the density of $D(\mathscr{A})$.
 
    \noindent
    {\it Disipativity: } Let $V = (v_1,v_2,v_3,v_4)^T \in D(\mathscr{A})$. Using Young's inequality, 
    for any $\varepsilon, \tilde{\varepsilon} > 0$, we easily get
    \begin{align*}
        \langle\mathcal{A} V &-\lambda V, V\rangle_{\mathscr{H}} = 
        \left\langle
        \begin{pmatrix} 
            v_2-\lambda v_1 \\ 
            c_1 \triangle v_1+c_2\triangle v_3|_{s=1}+d_1\triangle v_2+d_2\triangle v_4|_{s=1}-\lambda v_2 \\ 
            -(1/\tau) \partial_s v_3-\lambda v_3 \\ 
            -(1/\tau) \partial_s v_4-\lambda v_4 
        \end{pmatrix},
        \begin{pmatrix} 
            v_1 \\ 
            v_2 \\ 
            v_3 \\ 
            v_4 
        \end{pmatrix}
        \right\rangle_\mathscr{H} \\
        &= c_1\langle\nabla v_2,\nabla v_1\rangle-\lambda c_1 \|\nabla v_1\|^2-c_1\langle\nabla v_1,\nabla v_2\rangle-c_2\langle\nabla v_3|_{s=1},\nabla v_2\rangle-d_1 \|\nabla v_2\|^2 \\
        &-d_2\langle\nabla v_4|_{s=1},\nabla v_2\rangle-\lambda\|v_2\|^2-\frac{d_1 \xi}{2\tau}\int_0^1 e^{cs} \tau\partial_s\|\nabla v_3\|^2\mathrm{d}s-d_1\lambda\tau \xi\int_0^1 e^{cs} \|\nabla v_3\|^2\mathrm{d}s \\
        &-\frac{|d_2|\xi}{2\tau}\int_0^1e^{cs} \tau\partial_s\|\nabla v_4\|^2\mathrm{d}s-|d_2|\lambda\tau\xi\int_0^1 e^{cs} \|\nabla v_4\|^2\mathrm{d}s\\
        &= -\lambda c_1\|\nabla v_1\|^2-c_2\langle\nabla v_3|_{s=1},\nabla v_2\rangle-d_1 \|\nabla v_2\|^2-\lambda\|v_2\|^2-d_2\langle\nabla v_4|_{s=1},\nabla v_2\rangle \\
        &-\frac{d_1 e^{c} \xi}{2}\|\nabla v_3|_{s=1}\|^2 +\frac{d_1 \xi}{2}\|\nabla v_3|_{s=0}\|^2-\frac{|d_2| e^{c} \xi}{2}\|\nabla v_4|_{s=1}\|^2+\frac{|d_2|\xi}{2}\|\nabla v_4|_{s=0}\|^2 \\
        &+\xi(d_1 c/2-d_1\lambda\tau)\int_0^1 e^{cs}\|\nabla v_3\|^2\mathrm{d}s+\xi(|d_2|c/2-d_2\lambda\tau)\int_0^1 e^{cs}\|\nabla v_4\|^2\mathrm{d}s\\
        &\leq -\lambda c_1\|\nabla v_1\|^2+\frac{|c_2|}{2 \varepsilon}\|\nabla v_3|_{s=1}\|^2+\frac{|c_2| \varepsilon}{2}\|\nabla v_2\|^2-d_1 \|\nabla v_2\|^2-\lambda\|v_2\|^2 + \frac{|d_2|}{2 \tilde{\varepsilon}}\|\nabla v_4|_{s=1}\|^2 \\
        &+ \frac{|d_2| \tilde{\varepsilon}}{2}\|\nabla v_2\|^2-
        \frac{d_1 e^{c} \xi}{2}\|\nabla v_3|_{s=1}\|^2+\frac{d_1 \xi}{2}\|\nabla v_1\|^2-\frac{|d_2|e^{c} \xi}{2}\|\nabla v_4|_{s=1}\|^2+\frac{|d_2|\xi}{2}\|\nabla v_2\|^2 \\
        & + \xi ( d_1 c / 2 -d_1\lambda\tau)\int_0^1 e^{cs} \|\nabla v_3\|^2\mathrm{d}s + \xi ( |d_2|c/2-|d_2|\lambda\tau)\int_0^1 e^{cs} \|\nabla v_4\|^2\mathrm{d}s\\
        &= \Big(\frac{\xi d_1}{2}-\lambda c_1\Big)\|\nabla v_1\|^2 + \Big(\frac{|c_2| \varepsilon}{2}-d_1+\frac{|d_2| \tilde{\varepsilon}}{2}+\frac{|d_2|\xi}{2}\Big)\|\nabla v_2\|^2 + \Big(\frac{|c_2|}{2\varepsilon}-\frac{d_1 e^{c} \xi }{2}\Big)\|\nabla v_3|_{s=1}\|^2 \\
        &- \lambda\|v_2\|^2  +\Big(\frac{|d_2|}{2\tilde{\varepsilon}}-\frac{|d_2| e^{c} \xi}{2}\Big)\|\nabla v_4|_{s=1}\|^2 + \xi (c/2 - \lambda \tau)\Big(d_1\int_0^1 \|\nabla v_3\|^2\mathrm{d}s+|d_2|\int_0^1 \|\nabla v_4\|^2\mathrm{d}s\Big),
    \end{align*}
    where $\langle \cdot, \cdot\rangle$ and $\|\cdot\|$ denote the standard inner product and the norm on $L^{2}(\Omega)$, respectively.
    Taking $\varepsilon = \frac{d_1}{3|c_2|}$, $\tilde{\varepsilon} = \frac{d_1}{3|d_2|} $ and $\lambda > \max \{ \frac{c}{2\tau}, \frac{\xi d_1}{2c_1} \}$, we arrive at
    \begin{align*}
        \langle\mathcal{A} V - \lambda V, V\rangle_{\mathscr{H}} &\leq   
             \Big(\frac{|c_2| \varepsilon}{2}-d_1+\frac{|d_2| \tilde{\varepsilon}}{2}+\frac{|d_2|\xi}{2}\Big)\|\nabla v_2\|^2 + \Big(\frac{|c_2|}{2\varepsilon}-\frac{d_1 e^{c} \xi }{2}\Big)\|\nabla v_3|_{s=1}\|^2 \\
             &+\Big(\frac{|d_2|}{2\tilde{\varepsilon}}-\frac{|d_2| e^{c} \xi}{2}\Big)\|\nabla v_4|_{s=1}\|^2 \\
             & = \Big(\frac{d_1}{6}-d_1+\frac{d_1 }{6}+\frac{d_1}{6}\Big)\|\nabla v_2\|^2 + \Big(\frac{3c_2^2}{2 d_1}-\frac{d_1^2 e^{c}  }{6|d_2|}\Big)\|\nabla v_3|_{s=1}\|^2 \\
             &+\Big(\frac{3d_2^2}{2d_1}-\frac{d_1 e^{c} }{6}\Big)\|\nabla v_4|_{s=1}\|^2.  
    \end{align*}
    It easily follows from conditions (\ref{WELL_POS_CONDITIONS_ON_XI}) that  
    \begin{equation}
        \label{EQUATION_OPERATOR_A_MINUS_LAMBDA_DISSIPATIVE}
        \langle\mathcal{A} V-\lambda V, V\rangle_{\mathscr{H}} \leq 0 \quad \text{ or, equivalently } \quad
        \langle\mathcal{A} V, V\rangle \leq \lambda \|V\|_{\mathscr{H}}^{2}.
    \end{equation}

    \noindent
    {\it Surjectivity: } Let $\lambda > 0$ be a large number to be selected later. For any $M = (m_1, m_2 ,m_3, m_4)^T \in \mathscr{H}$, 
    we need to find $V = (v_1, v_2, v_3, v_4)^{T} \in D(\mathscr{A})$ such that
    \begin{equation}
        \notag
        (\lambda \operatorname{Id} - \mathscr{A}) V = M.
    \end{equation}
    In the component notation, this equation reads as
    \begin{align}
        \label{WELL_POS_EQ6}
        \lambda v_1-v_2 &= m_1, \\
        \label{WELL_POS_EQ7}
        \lambda v_2-c_1 \triangle v_1-c_2\triangle v_3|_{s=1}-d_1\triangle v_2-d_2\triangle v_4|_{s=1} &= m_2,\\
        \label{WELL_POS_EQ8}
        \lambda v_3+(1/\tau) \partial_s v_3 &= m_3,\\
        \label{WELL_POS_EQ9}
        \lambda v_4+(1/\tau) \partial_s v_4 &= m_4.
    \end{align}
    From Equation (\ref{WELL_POS_EQ6}), we easily get 
    \begin{equation}
        \label{WELL_POS_EQ1}
        v_2 = \lambda v_1-m_1.
    \end{equation}
    Solving Equations (\ref{WELL_POS_EQ8}), (\ref{WELL_POS_EQ9}) with respect to $v_3$ and $v_4$, respectively, we obtain
    \begin{align}
        \label{WELL_POS_EQ2}
        v_3(s,\cdot) &= e^{-\lambda \tau s}v_1+\tau\int_0^{s}e^{-\lambda \tau (s-\sigma)}m_3(\sigma, \cdot)\mathrm{d}\sigma, \\
        \label{WELL_POS_EQ3}
        v_4(s,\cdot) &= e^{-\lambda \tau s}v_2+\tau\int_0^{s}e^{-\lambda \tau (s-\sigma)}m_4(\sigma, \cdot)\mathrm{d}\sigma \text{ for } s \in [0, 1].
    \end{align}
    Next, from Equations (\ref{WELL_POS_EQ1})--(\ref{WELL_POS_EQ3}), we get
    \begin{align}
        \label{WELL_POS_EQ4}
        v_3(1,\cdot) &= e^{-\lambda \tau}v_1+\tau\int_0^1 e^{-\lambda \tau (1-\sigma)}m_3(\sigma, \cdot)\mathrm{d}\sigma, \\
        \label{WELL_POS_EQ5}
        v_4(1,\cdot) &= e^{-\lambda \tau}v_2+\tau\int_0^1e^{-\lambda \tau (1-\sigma)}m_4(\sigma, \cdot)\mathrm{d}\sigma \\
        \notag
        &= e^{-\lambda \tau}\lambda v_1-e^{-\lambda \tau}m_1+\tau\int_0^1e^{-\lambda \tau (1-\sigma)}m_4(\sigma, \cdot)\mathrm{d}\sigma.
    \end{align}
    Using Equations (\ref{WELL_POS_EQ1})--(\ref{WELL_POS_EQ5}), we can transform Equation (\ref{WELL_POS_EQ7}) as follows
    \begin{align*}
        \lambda^2 v_1 &- \lambda m_1-c_1 \triangle v_1-c_2e^{-\lambda \tau}\triangle v_1-c_2\tau\int_0^1 e^{-\lambda \tau (1-\sigma)}\triangle m_3(\sigma, \cdot)\mathrm{d}\sigma-d_1\lambda \triangle v_1 \\
        &+ d_1 \triangle m_1 -d_2e^{-\lambda \tau}\lambda \triangle v_1+d_2e^{-\lambda \tau}\triangle m_1-d_2\tau\int_0^1e^{-\lambda \tau (1-\sigma)}\triangle m_4(\sigma, \cdot)\mathrm{d}\sigma =m_2
    \end{align*}
    or, equivalently,
    \begin{align}
         \label{WELL_POS_Contin}
        \begin{split}
            \lambda^2 v_1 &- (c_1+c_2e^{-\lambda \tau}+d_1\lambda+d_2e^{-\lambda \tau}\lambda)\triangle v_1= \lambda m_1+m_2-(d_1+d_2e^{-\lambda \tau})\triangle m_1  \\
            &+ c_2\tau\int_0^1 e^{-\lambda \tau (1-\sigma)}\triangle m_3(\sigma, \cdot)\mathrm{d}\sigma+d_2\tau\int_0^1e^{-\lambda \tau (1-\sigma)}\triangle m_4(\sigma, \cdot)\mathrm{d}\sigma.
        \end{split}
    \end{align}

    Now, consider the Hilbert space $H_{\Gamma_0}^1(\Omega)$ and the bilinear form
    \begin{equation}
        \label{WELL_POS_BILINEAR_FORM_DEFINITION}
        \mathfrak{B}(v_1, \tilde{v}_1)=\int_{\Omega}\big(\lambda^2v_1\tilde{v}_1 + (c_1+c_2e^{-\lambda \tau}+d_1\lambda+d_2e^{-\lambda \tau}\lambda)\nabla v_1 \cdot \nabla\tilde{v}_1 \big)\mathrm{d}x.
    \end{equation}
    for $v_1, \tilde{v}_1 \in H_{\Gamma_0}^1(\Omega)$.
    By virtue of Green's formula (\ref{EQUATION_GREEN_FORMULA}),
    the variational formulation of Equation (\ref{WELL_POS_Contin}) reads as
    \begin{align}
        \label{WELL_POS_VARIATIONAL_PROBLEM}
        \mathfrak{B}(v_1, \tilde{v}_1) &= \int_{\Omega} (\lambda m_1+m_2)\tilde{v}_1  + (d_1+d_2e^{-\lambda \tau})\nabla m_1 \nabla \tilde{v}_1 \d x \\
        \notag
        & -\int_{\Omega} \big( c_2\tau\int_0^1 e^{-\lambda \tau (1-\sigma)}\nabla m_3(\sigma)\mathrm{d}\sigma+d_2\tau\int_0^1e^{-\lambda \tau (1-\sigma)}\nabla m_4(\sigma)\mathrm{d}\sigma \big)  \nabla \tilde{v}_1 \d x.
    \end{align}
    To solve this equation, we will exploit the Lemma of Lax \& Milgram. Using Poincar\'{e} \& Friedrichs' inequality and the positivity of $\lambda$, we estimate
    \begin{align*}
        \mathfrak{B}(v_1,v_1) &\geq \int _{\Omega} (c_1+c_2e^{-\lambda \tau}+d_1\lambda+d_2e^{-\lambda \tau}\lambda)\|\nabla v_1\|^2\mathrm{d}x \geq c_1 \int _{\Omega}\|\nabla v_1\|^2\mathrm{d}x \\
        &\geq \frac{c_1}{2} \|\nabla v_1\|^2 + \frac{c_1}{2} \|\nabla v_1\|^2 \geq \frac{c_1}{2c_p} \| v_1\|^2 + \frac{c_1}{2} \|\nabla v_1\|^2 
        \geq \frac{c_1}{2} \min \{c_p^{-1},1\} \|v_1\|^2_{H^1(\Omega)},
    \end{align*}
    where $c_p > 0$ is the Poincar\'{e} constant. This proves coercivity of the bilinear form $\mathfrak{B}$. Next, we verify continuity of $\mathfrak{B}$:
    \begin{align*}
        \mathfrak{B}(v_1, \tilde{v}_1) &= \lambda^2\int _{\Omega}v_1\tilde{v}_1\mathrm{d}x + (c_1+c_2e^{-\lambda \tau}+d_1\lambda+d_2e^{-\lambda \tau}\lambda)\int _{\Omega}\nabla v_1 \nabla\tilde{v}_1\mathrm{d}x \\
        &\leq \lambda^2\| v_1\|_{L^2}\| \tilde{v}_1\|_{L^2} + (c_1+c_2e^{-\lambda \tau}+d_1\lambda+d_2e^{-\lambda \tau}\lambda)\|\nabla v_1\|_{L^2}\|\nabla \tilde{v}_1\|_{L^2}\\
        & \leq \lambda^2 \|  v_1\|_{H^1(\Omega)}\|  \tilde{v}_1\|_{H^1(\Omega)} + (c_1+c_2e^{-\lambda \tau}+d_1\lambda+d_2e^{-\lambda \tau}\lambda)\| v_1\|_{H^1(\Omega)}\| \tilde{v}_1\|_{H^1(\Omega)} \\
        & \leq 2 \max \{\lambda^2, c_1+c_2e^{-\lambda \tau}+d_1\lambda+d_2e^{-\lambda \tau}\lambda\}\| v_1\|_{H^1(\Omega)}\| \tilde{v}_1\|_{H^1(\Omega)}.
    \end{align*}
    It can further be easily seen that right-hand side of Equation (\ref{WELL_POS_VARIATIONAL_PROBLEM}) constitutes a linear, bounded functional on $H_{\Gamma_0}^1$. 
    Since $\mathfrak{B}$ is a bilinear coercive form, using Lax \& Milgram's lemma, we infer the existence of a unique solution $v^* \in H_{\Gamma_0}^1(\Omega)$ to Equaiton (\ref{WELL_POS_VARIATIONAL_PROBLEM}). 
    Further, plugging $v_1$ into Equations (\ref{WELL_POS_EQ1})--(\ref{WELL_POS_EQ3}), we obtain $v_2,v_3,v_4$, respectively.

    There remains to check $(v_1,v_2,v_3,v_4)^{T} \in D(\mathscr{A})$. First, it follows from Equation (\ref{WELL_POS_EQ1}) that $v_2 \in H_{\Gamma_0}^1(\Omega)$. 
    Second, taking into account $m_3, m_4\in L^2\big(0, 1; H_{\Gamma_0}^1(\Omega)\big)$ and Equations (\ref{WELL_POS_EQ2})--(\ref{WELL_POS_EQ3}), we get $v_3,v_4 \in H^1\big(0, 1; H_{\Gamma_0}^1(\Omega)\big)$. 
    Third, it easily follows from Equations (\ref{WELL_POS_EQ2})--(\ref{WELL_POS_EQ3}) that $v_3|_{s=0} = v_1$, $v_4|_{s=0} = v_2$. 
    Finally, taking into account Equations (\ref{WELL_POS_EQ7}) and (\ref{WELL_POS_VARIATIONAL_PROBLEM}), we conclude 
    that the integral identity in Equation (\ref{WELL_POS_OPERATOR_DOMAIN_REWRITTEN}) holds true and, therefore,  $(v_1,v_2,v_3,v_4)^{T} \in D(\mathscr{A})$. 

    On the strength of Lumer-Phillips' Theorem \cite[Theorem 4.3, p. 14]{LiZh1999}), $\mathscr{A}-\lambda \, \mathrm{Id}$ is the infinitesimal generator of a $C_0$-semigroup of contractions on $\mathscr{H}$. 
    Since $\mathscr{A}$ is a bounded perturbation of $\mathscr{A}-\lambda \, \mathrm{Id}$, $\mathscr{A}$ generates a $C_0$-semigroup as well \cite[Chapter 3]{Pa1992}.
\end{proof}

\noindent
Now, by virtue of \cite[Theorem 1.3, p. 102]{Pa1992}, it follows that the abstract formulation (\ref{WELL_POS_ABSTRACT_FORMULATION}) of Equations (\ref{INTRODUCTION_EQUATION})--(\ref{INTRODUCTION_IC_2}) is Hadamard well-posed.
\begin{theorem}
    For any $V^0 \in \mathscr{H}$ and $F \in L^{2}_{\mathrm{loc}}\big(0, \infty; \mathscr{H}\big)$, 
    there exists a unique mild solution $V \in C^0\big([0, \infty), \mathscr{H}\big)$ to Equation (\ref{WELL_POS_ABSTRACT_FORMULATION})
    and constants $M, \omega > 0$ (independent of $V^{0}$ and $F$) such that
    \begin{equation}
        \max_{t \in [0, T]} \big\|V(t)\big\|_{\mathscr{H}} \leq M e^{\omega t} \|V_{0}\|_{\mathscr{H}} + M e^{\omega T} \|F\|_{L^{2}(0, T; \mathscr{H})}
        \text{ for } T > 0. \notag
    \end{equation}
    Moreover, if $V^0 \in D(\mathscr{A})$ and $F \in C^{1}\big([0, \infty), \mathscr{H}\big) \cup C^{0}\big([0, \infty), D(\mathscr{A})\big)$, the mild solution is classical:
    \begin{equation}
        \notag
        V \in C^1\big([0, \infty), \mathscr{H}\big) \cap C^0\big([0, \infty), D(\mathscr{A})\big).
    \end{equation}
\end{theorem}

\begin{corollary}
    \label{COROLLARY_EXISTENCE_AND_UNIQUENESS}
    For any initial data $y^{0} \in H^{1}_{\Gamma_{0}}(\Omega)$, $y^{1} \in L^{2}(\Omega)$, 
    $\varphi \in H^{1}\big({-\tau}, 0; H^{1}_{\Gamma_{0}}(\Omega)\big)$ and a right-hand side $f \in L^{2}_{\mathrm{loc}}\big(0, \infty; L^{2}(\Omega)\big)$,
    Equations (\ref{INTRODUCTION_EQUATION})--(\ref{INTRODUCTION_IC_2}) possess a unique mild solution
    \begin{equation}
        \notag
        y \in C^{2}\big([0, \infty), L^{2}(\Omega)\big) \cap C^{1}\big([0, \infty), H^{1}_{\Gamma_{0}}(\Omega)\big) \cap
        H^{1}\big({-\tau}, 0; H^{1}_{\Gamma_{0}}(\Omega)\big).
    \end{equation}
    If, moreover, $\triangle (c_{1} y^{0} + d_{1} y^{1}) \in L^{2}(\Omega)$, $\varphi \in H^{2}\big({-\tau}, 0; H^{1}_{\Gamma_{0}}(\Omega)\big)$ 
    such that $\varphi(0+, \cdot) = y^{0}$, $\partial_{t} \varphi(0+, \cdot) = y^{1}$ and $f \in C^{1}\big([0, \infty), L^{2}(\Omega)\big)$,
    the mild solution is classical:
    \begin{align*}
        &y \in C^{2}\big([0, \infty), L^{2}(\Omega)\big) \cap C^{1}\big([0, \infty), H^{1}_{\Gamma_{0}}(\Omega)\big) \cap
        H^{2}\big({-\tau}, 0; H^{1}_{\Gamma_{0}}(\Omega)\big) \\
        &\text{with } \triangle (a y + b \partial_{t} y) \in C^{0}\big([0, \infty), L^{2}(\Omega)\big) \text{ satisfying Equations (\ref{EQUATION_GREEN_FORMULA}).}
    \end{align*}
\end{corollary}

\subsection{Higher-Regularity Solutions}
\label{SECTION_SOLUTION_REGULARITY}

We now briefly discuss the regularity of solutions to Equations (\ref{INTRODUCTION_EQUATION})--(\ref{INTRODUCTION_IC_2}).
To this end, consider the intantaneous part of Equations (\ref{INTRODUCTION_EQUATION})--(\ref{INTRODUCTION_IC_2}):
\begin{align}
    \label{EQUATION_KELVIN_VOIGT_PDE}
    \partial_{tt} y(t, x) - a \triangle y(t, x) - b \partial_t \triangle y(t, x) &= f(t, x) \text{ in } (0, \infty) \times \Omega, \\
    \label{EQUATION_KELVIN_VOIGT_BC_1}
    y(t,x) &= 0 \text{ on }  \Gamma_0, \\
    \label{EQUATION_KELVIN_VOIGT_BC_2}
    \frac{\partial y(t,x)}{\partial \nu} &= 0 \text{ on }  \Gamma_1, \\
    \label{EQUATION_KELVIN_VOIGT_IC}
    y(0,\cdot) = y^0, \quad \partial_t y(0,\cdot) &= y^{1} \text{ in } \Omega.
\end{align}
Defining the space $\mathcal{H} := L^{2}(\Omega)$ endowed with the usual inner product
and the operator
\begin{equation*}
    \mathcal{A} := -\triangle \colon D(\mathcal{A}) \subset \mathcal{H} \to \mathcal{H}, \quad
    D(\mathcal{A}) = \big\{y \in H^{1}_{\Gamma_{0}}(\Omega) \,|\, \text{Green's formula (\ref{EQUATION_GREEN_FORMULAR_LAPLACE}) holds true}\big\}
\end{equation*}
with the Green's formula
\begin{equation}
    \label{EQUATION_GREEN_FORMULAR_LAPLACE}
    \int_{\Omega} \triangle u \cdot v \mathrm{d}x + \int_{\Omega} \nabla u \cdot \nabla v \mathrm{d}x = 0 \quad \text{ for } \quad v \in H^{1}_{\Gamma_{0}}(\Omega),
\end{equation}
Equations (\ref{EQUATION_KELVIN_VOIGT_PDE})--(\ref{EQUATION_KELVIN_VOIGT_IC}) can be written as
\begin{align}
    \label{EQUATION_KELVIN_VOIGT_OPERATOR_PDE}
    \partial_{tt} y(t, \cdot) + \mathcal{A} \big(a y(t, \cdot) + b \partial_t y(t, \cdot)\big) &= 0 \text{ for } t > 0 \\
    \label{EQUATION_KELVIN_VOIGT_OPERATOR_IC}
    y(0,\cdot) = y^0, \quad \partial_t y(0,\cdot) &= y^{1}.
\end{align}

By standard results, we know
\begin{equation}
    \notag
    D(\mathcal{A}^{1/2}) = H^{1}_{\Gamma_{0}}(\Omega), \; \big[D(\mathcal{A}^{1/2})\big]' = \big(H^{1}_{\Gamma_{0}}(\Omega)\big)'
    \text{ and } \mathcal{A} \in L\big(D(\mathcal{A}), \mathcal{H}\big) \cap L\big(D(\mathcal{A}^{1/2}),\big[D(\mathcal{A}^{1/2})\big]'\big),
\end{equation}
where we do not distinguish between $\mathcal{A}$ and its extension.
From \cite[Proposition 1]{KaLa2009}, we have:
\begin{theorem}
    \label{THEOREM_KALTENBACHER_LASIECKA}
    Suppose $f \in L^{2}_{\mathrm{loc}}(0, \infty; \mathcal{H}) \cap H^{1}_{\mathrm{loc}}\big(0, \infty; \big[D(\mathcal{A}^{1/2})\big]'\big)$
    and $y^{0} \in D(\mathcal{A})$, $y^{1} \in D(\mathcal{A}^{1/2})$ such that $f(0) - b \mathcal{A} y^{1} \in \mathcal{H}$.
    Then, Equations (\ref{EQUATION_KELVIN_VOIGT_OPERATOR_PDE})--(\ref{EQUATION_KELVIN_VOIGT_OPERATOR_IC}) possess a unique global solution
    \begin{equation*}
        y \in C\big([0, \infty), D(\mathcal{A})\big) \cap C^{1}\big([0, \infty), D(\mathcal{A}^{1/2})\big) \cap
        C^{2}\big([0, \infty), \mathcal{H}\big) \cap H^{2}_{\mathrm{loc}}\big(0, \infty; D(\mathcal{A}^{1/2})\big).
    \end{equation*}
\end{theorem}

\begin{theorem}
    \label{THEOREM_REGULARITY}
    Suppose $y^{0}, y^{1} \in D(\mathcal{A})$, $\varphi^{0} \in H^{2}({-\tau}, 0; H^{1/2}_{\Gamma_{0}}(\Omega)\big) \cap H^{1}\big({-\tau}, 0; D(\mathcal{A})\big)$
    with $\varphi(0+, \cdot) = y^{0}$, $\partial_{t} \varphi(0+, \cdot) = y^{1}$.
    Then the classical solution $y$ given in Corollary \ref{COROLLARY_EXISTENCE_AND_UNIQUENESS} satisfies
    \begin{align*}
        y &\in H^{2}_{\mathrm{loc}}\big(0, \infty; D(\mathcal{A})\big) \cap C^{1}\big([0, \infty), H^{1}_{\Gamma_{0}}(\Omega)\big) \cap C^{0}\big([0, \infty), D(\mathcal{A})\big) \cap \\
        &\hspace{0.2in} H^{2}\big({-\tau}, 0; H^{1/2}_{\Gamma_{0}}(\Omega)\big) \cap H^{1}\big({-\tau}, 0; D(\mathcal{A})\big).
    \end{align*}
\end{theorem}

\begin{proof}
    We employ the step method (cf. \cite{KhuPoRa2015}).
    Consider Equations (\ref{INTRODUCTION_EQUATION})--(\ref{INTRODUCTION_IC_2}) for $t \in (0, \tau)$.
    The equations reduce to (\ref{EQUATION_KELVIN_VOIGT_OPERATOR_PDE})--(\ref{EQUATION_KELVIN_VOIGT_OPERATOR_IC}) with 
    $a = c_{1}$, $b = d_{1}$ and $f(t) = c_{1} \triangle \varphi^{0}\big(t - \tau\big) + d_{1} \triangle \varphi^{1}\big(t - \tau\big)$.
    By assumption, $f \in L^{2}_{\mathrm{loc}}(0, \tau; \mathcal{H}) \cap H^{1}_{\mathrm{loc}}\big(0, \infty; \big(H^{1}_{\Gamma_{0}}(\Omega)\big)'\big)$.
    Hence, the conditions of Theorem \ref{THEOREM_KALTENBACHER_LASIECKA} are satisfied, and the desired regularity follows for $t \in [0, \tau]$.
    Same argument can be inductively repeated on $[\tau, 2\tau], [2\tau, 3\tau], \dots$,
    where the regularity of $f$ will follow from Corollary \ref{COROLLARY_EXISTENCE_AND_UNIQUENESS} as well as the induction assumption.
    Using \cite[Lemma 6.1]{KhuPoRa2015}, the desired regularity can be extended to $[0, \infty)$.
\end{proof}

\begin{remark}
    On each of the intervals $\big[(k - 1) \tau, k \tau\big]$, $k \in \mathbb{N}$, 
    by virtue of Theorem \ref{THEOREM_KALTENBACHER_LASIECKA}, we have $y_{[(k - 1) \tau, k \tau]} \in C^{2}\big([0, \infty), L^{2}(\Omega)\big)$.
    Unfortunately, as it is usually the case with delay differential equations or neutral equations, in general, this regularity does not extend to $[0, \infty)$.
\end{remark}

As for the `elliptic' regularity of $D(\mathcal{A})$, assuming $\partial \Omega$ is $C^{2}$, $\bar{\Gamma}_{0} \cap \bar{\Gamma}_{1} = \emptyset$, by standard elliptic results,
\begin{equation}
    D(\mathcal{A}) = \Big\{y \in H^{2}(\Omega) \cap H^{1}_{\Gamma_{0}}(\Omega) \,\big|\, \frac{\partial y}{\partial \nu}\big|_{\Gamma_{1}} = 0\Big\} \notag
\end{equation}
and the graph norm $\|\cdot\|_{D(\mathcal{A})} = \|\triangle \cdot\|_{L^{2}(\Omega)}$ is equivalent with $\|\cdot\|_{H^{2}(\Omega)}$.

\subsection{Well-Posedness for Radon-Measure Right-Hand Sides}
\label{SECTION_RADON_MEASURE_SOLUTION}

In contrast to Section \ref{SECTION_SOLUTION_REGULARITY}, we now seek to establish a solution theory to equations with Radon-measure distributional right-hand sides
modeling impulse impacts. This theory will later prove to be crucial in Section \ref{SECTION_NUMERICAL_EXAMPLE}.
With start with a general framework to later apply our general result to Equations (\ref{INTRODUCTION_EQUATION})--(\ref{INTRODUCTION_IC_2}).

Let $\mathscr{H}$ be a Hilbert space (thus, possessing Radon-Nikodym property). Further, let $\mathscr{A} \colon D(\mathscr{A}) \subset \mathscr{H} \to \mathscr{H}$ 
be a linear, closed, densely defined operator such that $\mathscr{A} - \lambda \operatorname{Id}$ is maximally dissipative for some $\lambda > 0$
(cf. Equation (\ref{EQUATION_OPERATOR_A_MINUS_LAMBDA_DISSIPATIVE} for our delayed Kelvin \& Voigt operator).
Then, on the strength of Lumer \& Phillips theorem and standard perturbation results, $\mathscr{A}$ generates a $C_{0}$-semigroup $\{e^{tA}\}_{t \geq 0}$ on $\mathscr{H}$.
Using standard extrapolation techniques, the operator $\mathscr{A}$ can be extended to a bounded linear operator from $\mathscr{H}$ to $\big[D(\mathscr{A}^{\ast})\big]'$
generating a $C_{0}$-semigroup on $\big[D(\mathscr{A}^{\ast})\big]'$, which, in turn, is a continuous extension of $\{e^{t\mathscr{A}}\}_{t \geq 0}$.
In the sequel, we will not distinguish between $\mathscr{H}$- and $\big[D(\mathscr{A}^{\ast})\big]'$-versions of the operators.

For $F \colon [a, b] \to \mathscr{H}$ with $\infty < a < b < \infty$, consider the total variation functional
\begin{equation}
    \notag
    \|F\|_{\mathrm{TV}([a, b], \mathscr{H})} :=
    \sup\Big\{\sum_{n = 1}^{n} \big\|F(t_{k}) - F(t_{k - 1})\big\|_{\mathscr{H}} \,\big|\, a = t_{0} < t_{1} < \dots < t_{n} = b, \quad n \in \mathbb{N}\Big\}
\end{equation}
and introduce the bounded variation space $\mathrm{BV}\big([a, b], \mathscr{H}\big)$ consisting of all functions $F \colon [a, b] \to \mathscr{H}$ with a finite BV-norm
\begin{equation}
    \notag
    \|F\|_{\mathrm{BV}([a, b], \mathscr{H})} := \big\|F(a)\big\|_{\mathscr{H}} + \|F\|_{\mathrm{TV}([a, b], \mathscr{H})}.
\end{equation}
With the usual metrisation approach, the local BV space $\mathrm{BV}_{\mathrm{loc}}\big([0, \infty), \mathscr{H}\big)$ can also be introduced.

Let $F \in \mathrm{BV}\big([0, T], \mathscr{H}\big)$ be a c\'{a}dl\'{a}g function (right-continuous with left limits)
-- being a `good representation' of its equivalence class in the sense of \cite[Theorem 3.28]{AmFuPa2000}.
Then, there exists an $\mathscr{H}$-valued Radon measure $DF$ on the Borel $\sigma$-algebra $\mathcal{B}\big([0, T])$ of $[0, T]$ such that
\begin{equation}
    \label{EQUATION_RADON_FUNDAMENTAL_THEOREM_OF_CALCULUS}
    F(t) = F(s) + DF\big((s, t]\big) = F(s) + \int_{(s, t]} \mathrm{d} (DF)(\xi) =  F(s) + \int_{(s, t]} \mathrm{d}F(\xi) \quad \text{ for } 0 \leq s \leq t \leq T,
\end{equation}
where the second last integral is a Bochner integral and the last integral is a Bochner-Stieltjes integral.
Note that $\big\|DF\big(\{0\}\big)\big\|_{\mathscr{H}} = 0$ since $F$ is right-continuous at $0$.
Letting $\mathscr{M}\big([0, T], \mathscr{H}\big)$ denote the space of $\mathscr{H}$-valued Radon measures on $\mathcal{B}\big([0, T]\big)$ 
equipped with the total variation norm for measures 
\begin{equation}
    \notag
    \|\mu\|\big([0, T]\big) := \sup\Big\{\sum_{n = 1}^{\infty} \big\|\mu(A_{n})\big\|_{\mathscr{H}} \,\big|\,
    A_{1}, A_{2}, \dots, A_{n} \text{ is a partition of } [0, T], \quad n \in \mathbb{N}\Big\},
\end{equation}
Equation (\ref{EQUATION_RADON_FUNDAMENTAL_THEOREM_OF_CALCULUS}) is a generalization of the fundamental theorem of calculus
and establishes an isomorphism between $\mathscr{H} \times \mathscr{M}\big([0, T], \mathscr{H}\big)$
and the subspace of c\'{a}dl\'{a}g functions in $\mathrm{BV}([0, T], \mathscr{H})$ via
\begin{equation*}
    \big(F(0), DF\big) \mapsto F.
\end{equation*}

For a c\'{a}dl\'{a}g function $F \in \mathrm{BV}\big([a, b], \mathscr{H}\big)$, consider the following abstract operator equation
\begin{align}
    \label{EQUATION_CAUCHY_PROBLEM_RADON_MEASURE}
    DV = \mathscr{A} V + DF \text{ in } \mathscr{M}\big([0, T], \big[D(\mathscr{A}^{\ast})\big]'\big), \quad V(0) = V^{0} \text{ in } \mathscr{H}.
\end{align}
where the `density' $\mathscr{A} V$ induces a $\big(D(\mathscr{A}^{\ast})\big)'\big)$-valued Radon measure $A \mapsto \int_{A} \mathscr{A} V(t) \mathrm{d}t$ for Borel sets $A$.

\begin{definition}
    A c\'{a}dl\'{a}g function $V \in \mathrm{BV}([0, T], \mathscr{H})$ is called an extrapolated BV-solution to Equation (\ref{EQUATION_CAUCHY_PROBLEM_RADON_MEASURE})
    if it satisfies Equation (\ref{EQUATION_CAUCHY_PROBLEM_RADON_MEASURE}) in the sense of $\big[D(\mathscr{A}^{\ast})\big]'$-valued Radon measures on $\mathcal{B}\big([0, T]\big)$.
\end{definition}

\begin{theorem}
    \label{THEOREM_CAUCHY_PROBLEM_RADON_MEASURE_EXISTENCE}
    
    In addition to the assumptions above, let $\big\|DF\big(\{0\}\big)\big\|_{\mathscr{H}} = 0$.
    The function
    \begin{equation}
        \label{EQUATION_CAUCHY_PROBLEM_RADON_MEASURE_EXISTENCE}
        V(t) = e^{\mathscr{A} t} V^{0} + \int_{(0, t]} e^{\mathscr{A} (t - s)} \mathrm{d}F(s) \quad \text{ for } t \in [0, T] \notag
    \end{equation}
    is a unique extrapolated BV-solution to Equation (\ref{EQUATION_CAUCHY_PROBLEM_RADON_MEASURE}).
    Moreover, there exist $M, \omega > 0$ such that
    \begin{equation}
        \label{EQUATION_CAUCHY_PROBLEM_RADON_MEASURE_ESIMATE}
        \big\|V\|_{C^{0}([0, T], \mathscr{H})} \leq M e^{\omega t} \big(\|V^{0}\|_{\mathscr{H}} + \|F\|_{\mathrm{BV}([a, b], \mathscr{H})}\big). \notag
    \end{equation}
\end{theorem}

\begin{proof}
    {\it Existence: } We first verify the function $V$ defined in (\ref{EQUATION_CAUCHY_PROBLEM_RADON_MEASURE_EXISTENCE}) solves Equation (\ref{EQUATION_CAUCHY_PROBLEM_RADON_MEASURE}).
    Trivially, $V$ is a c\'{a}dl\'{a}g function satisfying $V(0) = V^{0}$.
    Computing
    \begin{align*}
        DV(A) &= \int_{A} \mathscr{A} e^{t\mathscr{A}} V^{0} \mathrm{d}t + DF(A) + \int_{A} \int_{(0, t]} \mathscr{A} e^{\mathscr{A} (t - s)} \mathrm{d}F(s) \mathrm{d}t, \\
        \mathscr{A} V(A) &= \int_{A} \mathscr{A} e^{\mathscr{A} t} V^{0} \mathrm{d}t + \int_{A} \int_{(0, t]} \mathscr{A} e^{\mathscr{A} (t - s)} \mathrm{d}F(s) \mathrm{d}t,
    \end{align*}
    we get $DV(A) - \mathscr{A} V(A) \equiv DF(A)$ for any Borel set $A \subset [0, T]$.
    
    {\it Uniqueness and continuous dependence: }
    Let $V(t)$ be an extrapolated solution to Equation (\ref{EQUATION_CAUCHY_PROBLEM_RADON_MEASURE}). Define
    \begin{equation}
        \notag
        \mathcal{E}(t) := \frac{1}{2} \max_{0 \leq s \leq t} \big\|V(s)\big\|_{\mathscr{H}}^{2}.
    \end{equation}
    Integrating $\langle V(s), \cdot\rangle_{\mathscr{H}}$ with respect to the Radon measure $\mu := DV - \mathscr{A} V + DF$, we get
    \begin{align*}
        0 &= \int_{[0, t]} \langle V(s), \mathrm{d}\mu(s)\rangle_{\mathscr{H}} \\
        &= \int_{[0, t]} \langle V(s), \mathrm{d}V(s)\rangle_{\mathscr{H}} - \int_{[0, t]} \langle \mathscr{A} V(s), V(s)\rangle_{\mathscr{H}} \mathrm{d}s
        - \int_{[0, t]} \langle V(s), \mathrm{d}F(s)\rangle_{\mathscr{H}} \\
        &= \frac{1}{2} \big\|V(t)\big\|_{\mathscr{H}}^{2} -  \frac{1}{2} \big\|V(0)\big\|_{\mathscr{H}}^{2} -
        \int_{0}^{t} \langle \mathscr{A} V(t), V(t)\rangle_{\mathscr{H}} \mathrm{d}s - \int_{[0, t]} \langle V(s), \mathrm{d}F(s)\rangle_{\mathscr{H}}
    \end{align*}
    Thus, we arrive at
    \begin{align*}
        \mathcal{E}(t) &\leq \mathcal{E}(0) + 2 \lambda \int_{0}^{t} \mathcal{E}(s) \mathrm{d}s + \big(\mathcal{E}(t)\big)^{1/2} \|DF\|\big([0, T]\big) \\
        &\leq \mathcal{E}(0) + 2 \lambda \int_{0}^{t} \mathcal{E}(s) \mathrm{d}s + \frac{1}{2} \big(\mathcal{E}(t)\big)^{1/2} \|F\|_{\mathrm{BV}([a, b], \mathscr{H})} \\
        &\leq \mathcal{E}(0) + 2 \lambda \int_{0}^{t} \mathcal{E}(s) \mathrm{d}s + \frac{\varepsilon}{2} \mathcal{E}(t)
        + \frac{1}{2\varepsilon} \|F\|_{\mathrm{BV}([a, b], \mathscr{H})}^{2},
    \end{align*}
    where we used Young's inequality. Selecting $\varepsilon$ sufficiently small and applying Gronwall's inequality, we get
    \begin{equation}
        \mathcal{E}(t) \leq \tilde{M} e^{2 \tilde{\omega} t} \Big(\mathcal{E}(0) + \|F\|_{\mathrm{BV}([a, b], \mathscr{H})}^{2}\Big)
    \end{equation}
    for some $\tilde{M}, \tilde{\omega} > 0$. Taking the square root and using the finite-dimensional norm equivalence,
    we obtain the desired estimate (\ref{EQUATION_CAUCHY_PROBLEM_RADON_MEASURE_ESIMATE}).
    The uniqueness now trivially follows due to linearity of Equation (\ref{EQUATION_CAUCHY_PROBLEM_RADON_MEASURE}) with the estimate applied to the difference of two solutions.
\end{proof}

Returning to our delayed Kelvin \& Voigt viscoelastic Equations (\ref{INTRODUCTION_EQUATION})--(\ref{INTRODUCTION_IC_2}), we get:
\begin{corollary}
    \label{COROLLARY_CAUCHY_PROBLEM_RADON_MEASURE_EXISTENCE}

    Let $f \in \mathscr{M}\big([0, T], L^{2}(\Omega)\big)$ be a Radon measure with $\big\|f\big(\{0\}\big)\big\|_{L^{2}(\Omega)} = 0$,
    $y^{0} \in H^{1}_{\Gamma_{0}}(\Omega)$, $y^{1} \in L^{2}(\Omega)$, $\varphi \in H^{1}\big({-\tau}, 0; L^{2}(\Omega)\big)$.
    Then, Equations (\ref{INTRODUCTION_EQUATION})--(\ref{INTRODUCTION_IC_2}) possess a unique extrapolated BV solution $y$ with
    $y \in \operatorname{BV}\big([0, T], H^{1}_{\Gamma_{0}}(\Omega)\big)$ and $\partial_{t} y \in \operatorname{BV}\big([0, T], L^{2}(\Omega)\big)$.
\end{corollary}

\section{Long-Time Behavior} \label{SECTION_LONG_TIME_BEHAVIOR}
Consider the `natural' energy 
\begin{equation}
    \label{STABILITY_ENERGY_DEFINITION}
    \mathcal{E}(t) = \frac{1}{2} \|\partial_t y\|^2 + \frac{c_1}{2}\|\nabla y\|^2 
    + \frac{\tau d_1}{2}\int_0^1 \|\nabla z(s,\cdot)\|^2\mathrm{d}s + \frac{\tau d_2}{2}\int_0^1 \|\partial_t \nabla z(s,\cdot)\|^2\mathrm{d}s.
\end{equation}
Throughout this section, let $c_p > 0$ denote the Poincar\'{e} \& Friechrichs' constant, i.e., $\|u\|^2 \le c_p \| \nabla u \|^2$ for every $u\in H^1_{\Gamma_0}(\Omega)$.

\begin{assumption}[Coefficients $c_1,c_2,d_1,d_2$]
    \label{ASSUMPTION_COEFFICIENTS}

    Suppose the coefficients $c_1,c_2,d_1,d_2$ satisfy:
    \begin{enumerate}
        \item $c_1 > 6 c_2 > 0$,

        \item $d_1^2 \ge \max\left\{ \frac{9c_1^2d_2^2}{c_1^2-9c_2^2}, \frac{18 c_1 c_2^2 c_p}{c_1^2-36c_2^2} \right\}$, $d_{2} > 0$.
    \end{enumerate}
\end{assumption}

\begin{remark} \label{STABILITY_REMARK_1}
    \begin{enumerate}
    	\item It follows from Assumption \ref{ASSUMPTION_COEFFICIENTS}.2 that $d_1 > 3d_2$.
    	
    	\item In contrast to \cite{AmNiPi1995}, our Assumption \ref{ASSUMPTION_COEFFICIENTS} does not impose a smallness condition on $\tau$.
    \end{enumerate}
\end{remark}

\subsection{Exponential Stability} \label{SUBSECTION_STABILITY}

\begin{theorem} \label{THEOREM_LYAPUNOV}
    Let $V^0 \in \mathscr{H}$. Under Assumption \ref{ASSUMPTION_COEFFICIENTS}, there exist constants $\tilde{\alpha}, C > 0$ such that
    \begin{equation*}
        \mathcal{E}(t) \leq Ce^{-2 \tilde{\alpha} t} \mathcal{E}(0).
    \end{equation*}
\end{theorem}

\begin{proof}
    Without loss of generality, recalling $D(\mathscr{A})$ is dense in $\mathscr{H}$, let $V^0 \in D(\mathscr{A})$. 
    Let $y$ denote the classical solution to Equations (\ref{WELL_POS_INITIAL_SYSTEM_EQUATION})--(\ref{WELL_POS_INITIAL_SYSTEM_IC_2}) for the initial datum $V^0$.

    Consider the Lyapunov functional
    \begin{align*}
        \mathcal{F}(t) &= \mathcal{F}_1(t) + \mathcal{F}_2(t) + \mathcal{F}_3(t) +\mathcal{F}_4(t) +\mathcal{F}_5(t), \text{ where } \\
        \mathcal{F}_1(t) &= \frac{N}{2}\|\partial_t y\|^2, \quad \mathcal{F}_2(t) = \frac{c_1N}{2}\|\nabla y\|^2, \quad \mathcal{F}_3(t) =  M \langle y, \partial_{t} y\rangle, \\
        \mathcal{F}_4(t) &= \frac{\tau \xi_1}{2}\int_0^1 \|\nabla z(s,\cdot)\|^2\mathrm{d}s, \quad \mathcal{F}_5(t) = \frac{\tau \xi_2}{2}\int_0^1 \|\partial_t \nabla z(s,\cdot)\|^2\mathrm{d}s
    \end{align*}
    and $M,N,\xi_1,\xi_2$ are positive constants to be fixed later.
    Computing the derivative of $\mathcal{F}$, using the generalized Young's inequality and Poincar\'{e} \& Friedrichs' inequality, we obtain
    \begin{align*}
        \dot{\mathcal{F}}_1(t) &= \partial_{t} \Big(\frac{N}{2}\|\partial_t y\|^2\Big) = N\langle \partial_{t}y, \partial_{tt} y\rangle\\
        &=N c_1\langle \partial_{t}y,  \triangle y\rangle + N c_2\langle \partial_{t}y,  \triangle z(1,\cdot) \rangle +N d_1\langle \partial_{t}y, \partial_{t} \triangle y\rangle +N d_2\langle \partial_{t}y,  \partial_{t} \triangle z(1,\cdot)\rangle\\
        &=-N c_1\langle \partial_{t} \nabla y,  \nabla y\rangle - N c_2\langle \partial_{t} \nabla y,  \nabla z(1,\cdot) \rangle -N d_1\langle \partial_{t} \nabla y, \partial_{t} \nabla y\rangle -N d_2\langle \partial_{t} \nabla y,  \partial_{t} \nabla z(1,\cdot)\rangle\\
        &\leq -N c_1\langle \partial_{t} \nabla y,  \nabla y\rangle +\frac{N c_2 \varepsilon_1}{2}\| \partial_{t} \nabla y\|^2 + \frac{N c_2}{2\varepsilon_1}\|\nabla z(1,\cdot) \|^2 - N d_1\|\partial_{t} \nabla y\|^2 \\
        &+ \frac{N d_2\varepsilon_2}{2}\| \partial_{t} \nabla y\|^2    + \frac{N d_2}{2\varepsilon_2} \|\partial_{t} \nabla z(1,\cdot)\|^2, \\
        \dot{\mathcal{F}}_2(t) &= \partial_{t} \Big(\frac{c_1N}{2}\|\nabla y\|^2\Big) = N c_1\langle  \nabla y, \partial_{t} \nabla y\rangle, \\
        \dot{\mathcal{F}}_3(t) &= \partial_{t} \Big(M \langle y, \partial_{t} y\rangle \Big) =  M \langle \partial_{t} y, \partial_{t} y\rangle + M \langle y, \partial_{tt} y\rangle\\
        &=M\|\partial_{t}y\|^2 - M c_1\langle \nabla y,  \nabla y\rangle - M c_2\langle \nabla y,  \nabla z(1,\cdot) \rangle - M d_1\langle \nabla y, \partial_{t} \nabla y\rangle -M d_2\langle \nabla y,  \partial_{t} \nabla z(1,\cdot)\rangle \\
        & \leq Mc_p\|\partial_{t}\nabla y\|^2 - M c_1\|\nabla y\|^2 + \frac{M c_2\varepsilon_3}{2}\|\nabla y\|^2 + \frac{M c_2}{2\varepsilon_3} \|\nabla z(1,\cdot)\|^2 \\
        &+ \frac{M d_1\varepsilon_4}{2}\|\nabla y\|^2 + \frac{M d_1}{2\varepsilon_4}\|\partial_t \nabla y\|^2 + \frac{M d_2\varepsilon_5}{2}\|\nabla y\|^2 + \frac{M d_2}{2\varepsilon_5} \|\partial_{t} \nabla z(1,\cdot)\|^2.
    \end{align*}
    Further, we evaluate the derivative of $\mathcal{F}_4$:
    \begin{align*}
        \dot{\mathcal{F}}_4(t) &= \partial_{t} \Big(\frac{\tau \xi_1}{2} \int_0^1 \|\nabla z(s,\cdot)\|^2\mathrm{d}s\Big) = \tau \xi_1 \int_0^1 \langle \nabla z(s,\cdot), \partial_t \nabla z(s,\cdot) \rangle \mathrm{d}s.
    \end{align*}
    On the strength of Equation (\ref{WELL_POS_INITIAL_SYSTEM_Z_1}), the latter integral is equal to
    \begin{align*}
        \dot{\mathcal{F}}_4(t) &= \tau \xi_1 \int_0^1 \big\langle \nabla z(s,\cdot), 
        -\frac{1}{\tau}\partial_s \nabla z(s,\cdot) \big\rangle \mathrm{d}s = -\frac{\xi_1}{2} \int_0^1 \partial_s \big( \|\nabla z(s,\cdot)\|^2 \big) \mathrm{d}s = -\frac{\xi_1}{2} \left.  \|\nabla z(s,\cdot)\|^2  \right|_{s=0}^1 \\
        &= -\frac{\xi_1}{2}\big(\| \nabla z(1,\cdot)\|^2-\| \nabla z(0,\cdot)\|^2\big).
    \end{align*}
    In the same fashion,
    \begin{align*}
        \dot{\mathcal{F}}_5(t) &= \partial_{t} \Big(\frac{\tau}{2}\int_0^1 \rho_2(\tau s)\|\partial_t \nabla z(s,\cdot)\|^2\mathrm{d}s\Big) = -\frac{\xi_2}{2}\big(\|\partial_t \nabla z(1,\cdot)\|^2-\|\partial_t \nabla z(0,\cdot)\|^2\big).
    \end{align*}
    Therefore, we get
    \begin{align}
        \notag
        \mathcal{F}'(t)&\leq \Big(-Mc_1+ \frac{M c_2\varepsilon_3}{2}+ \frac{M d_1\varepsilon_4}{2}+ \frac{M d_2\varepsilon_5}{2}+\frac{\xi_1}{2}\Big)\|\nabla y\|^2\\
        \label{STABILITY_LYAPUNOV_DERIVATIVE_ESTIMATION}
        &+ \Big(\frac{Nc_2\varepsilon_1}{2}-Nd_1+\frac{Nd_2\varepsilon_2}{2}+ Mc_p + \frac{M d_1}{2\varepsilon_4}+\frac{\xi_2}{2}\Big)\|\partial_t \nabla y\|^2 \\
        \notag
        & +\Big(\frac{Nc_2}{2\varepsilon_1}+\frac{Mc_2}{2\varepsilon_3}-\frac{\xi_1}{2}\Big)\|\nabla z(1,\cdot)\|^2 +\Big(\frac{Nd_2}{2\varepsilon_2}+\frac{Md_2}{2\varepsilon_5}-\frac{\xi_2}{2}\Big)\|\partial_t \nabla z(1,\cdot)\|^2.
    \end{align}
    Next, we will choose $N, M, \xi_1, \xi_2$ such that the following conditions hold:
    \begin{align}
        \label{STABILITY_NEGATIVNESS_COEF_1}
        -Mc_1+ \frac{M c_2\varepsilon_3}{2}+ \frac{M d_1\varepsilon_4}{2}+ \frac{M d_2\varepsilon_5}{2}+\frac{\xi_1}{2} &< 0,\\
        \label{STABILITY_NEGATIVNESS_COEF_2}
        \frac{Nc_2\varepsilon_1}{2}-Nd_1+\frac{Nd_2\varepsilon_2}{2}+ Mc_p + \frac{M d_1}{2\varepsilon_4}+\frac{\xi_2}{2} &< 0,\\
        \label{STABILITY_NEGATIVNESS_COEF_3}
        \frac{Nc_2}{2\varepsilon_1}+\frac{Mc_2}{2\varepsilon_3}-\frac{\xi_1}{2} &< 0,\\
        \label{STABILITY_NEGATIVNESS_COEF_4}
        \frac{Nd_2}{2\varepsilon_2}+\frac{Md_2}{2\varepsilon_5}-\frac{\xi_2}{2} &< 0.
    \end{align}
    Letting
    \begin{align*}
        \varepsilon_1 = \frac{d_1}{3c_2}, \quad 
        \varepsilon_2 = \frac{d_1}{3d_2}, \quad 
        \varepsilon_3 = \frac{c_1}{3c_2}, \quad 
        \varepsilon_4 = \frac{c_1}{3d_1}, \quad  
        \varepsilon_5 = \frac{c_1}{3d_2}, \quad 
        \xi_1 = \frac{Mc_1}{3}, \quad
        \xi_2 = \frac{Nd_1}{3}.
    \end{align*}
    Equations (\ref{STABILITY_NEGATIVNESS_COEF_1}) -- (\ref{STABILITY_NEGATIVNESS_COEF_4}) can be expressed as
    \begin{align}
        \label{STABILITY_NEGATIVNESS_COEF_5}
        Mc_1/3 &< 0,\\
        \label{STABILITY_NEGATIVNESS_COEF_6}
        Mc_p + 3M d_1^2/(2 c_1) &< Nd_1/2,\\
        \label{STABILITY_NEGATIVNESS_COEF_7}
        3Nc_2^2/(2d_1) + 3Mc_2^2/(2c_1) &< Mc_1/6,\\
        \label{STABILITY_NEGATIVNESS_COEF_8}
        3Nd_2^2/(2 d_1) + 3Md_2^2/(2 c_1) &< Nd_1/6.
    \end{align}
    Observing the condition in Equation (\ref{STABILITY_NEGATIVNESS_COEF_5}) is satisfied. 
    Further, using Remark \ref{STABILITY_REMARK_1}, we rewrite Equations (\ref{STABILITY_NEGATIVNESS_COEF_6})--(\ref{STABILITY_NEGATIVNESS_COEF_8}) as follows
    \begin{align}
        \label{STABILITY_NEGATIVNESS_COEF_9}
        \frac{N}{M} & < \frac{d_1(c_1^2-9c_2^2)}{9c_1c_2^2},
        \\
        \label{STABILITY_NEGATIVNESS_COEF_10}
        \frac{N}{M} & > \frac{9d_2^2d_1}{c_1(d_1^2-9d_2^2)},
        \\
        \label{STABILITY_NEGATIVNESS_COEF_11}
        \frac{N}{M} & > \frac{2c_pc_1 +3d_1^2}{c_1d_1}.
    \end{align}
    On the other hand, Assumption \ref{ASSUMPTION_COEFFICIENTS} implies
    \begin{align*}
        (c_1^2-9c_2^2)(d_1^2-9d_2^2) > 81 d_2^2 c_2^2, \quad
        (c_1^2-9c_2^2) d_1^2 &> 9 c_2^2 (2 c_p c_1 + 3d_1^2),
    \end{align*}
    and, therefore,
    \begin{align*}
        \frac{(c_1^2-9c_2^2)d_1}{9 c_1 c_2^2} > \frac{9d_1d_2^2}{c_1(d_1^2-9d_2^2)}, \quad
        \frac{(c_1^2-9c_2^2)d_1}{9 c_1 c_2^2} > \frac{2 c_p c_1 + 3d_1^2}{c_1d _1}.
    \end{align*}
    The latter inequalities mean that there exist positive numbers $N, M$ such that the conditions in Equation (\ref{STABILITY_NEGATIVNESS_COEF_9})--(\ref{STABILITY_NEGATIVNESS_COEF_11}) hold.
    Taking into account  Equations (\ref{STABILITY_NEGATIVNESS_COEF_1})--(\ref{STABILITY_NEGATIVNESS_COEF_4}), 
    Equation (\ref{STABILITY_LYAPUNOV_DERIVATIVE_ESTIMATION}) furnishes the existence of $\alpha > 0$ such that
    \begin{equation}
        \label{STABILITY_LYAPUNOV_DERIVATIVE_ESTIMATION_2}
        \dot{\mathcal{F}}(t) \leq -2 \alpha \mathcal{E}(t) \text{ for } t \geq 0.
    \end{equation}

    Next, we prove $\mathcal{F}(\cdot)$ is equivalent with the energy $\mathcal{E}(\cdot)$, i.e., there exist constants $k_1, k_2 > 0$ such that
    \begin{equation}
        \label{STABILITY_LYAPUNOV_ESTIMATIONS}
        k_1 \mathcal{E}(t) \leq \mathcal{F}(t) \leq k_2 \mathcal{E}(t) \text{ for } t \geq 0.
    \end{equation}
    Letting $\hat{\varepsilon} = \sqrt{\frac{c_p}{c_1}}$ and applying the generalized Young's inequality and Poincar\'{e} \& Friedrichs' inequality, we get
    \begin{equation*}
        \left|\langle y, \partial_{t} y\rangle \right| \le 
        \frac{\hat{\varepsilon}}{2} \| \partial_t y \|^2 + \frac{1}{2\hat{\varepsilon}} \| y \|^2 \le  \frac{\hat{\varepsilon}}{2} \| \partial_t y \|^2 + \frac{c_p}{2\hat{\varepsilon}} \| \nabla y \|^2.
    \end{equation*}
    Therefore,
    \begin{align}
        \notag
        \mathcal{F}(t) &\leq \Big(\frac{N}{2}+\dfrac{M\hat{\varepsilon}}{2} \Big)\|\partial_t y\|^2+\Big(\frac{c_1N}{2}+\dfrac{Mc_p}{2\hat{\varepsilon}} \Big)\|\nabla y\|^2 \\ 
        \label{STABILITY_LYAPUNOV_ESTIMATION_1}
        &+\frac{\tau \xi_1}{2}\int_0^1 \|\nabla z(s,\cdot)\|^2\mathrm{d}s + \frac{\tau \xi_2}{2}\int_0^1 \|\partial_t \nabla z(s,\cdot)\|^2\mathrm{d}s, \\
        \notag
        \mathcal{F}(t) &\geq \Big(\frac{N}{2}-\dfrac{M\hat{\varepsilon}}{2} \Big)\|\partial_t y\|^2+\Big(\frac{c_1N}{2}-\dfrac{Mc_p}{2\hat{\varepsilon}} \Big)\|\nabla y\|^2 \\
        \label{STABILITY_LYAPUNOV_ESTIMATION_2}
        &+\frac{\tau \xi_1}{2}\int_0^1 \|\nabla z(s,\cdot)\|^2\mathrm{d}s + \frac{\tau \xi_2 }{2}\int_0^1 \|\partial_t \nabla z(s,\cdot)\|^2\mathrm{d}s.
    \end{align}
    On the strength of Equation (\ref{STABILITY_NEGATIVNESS_COEF_11}), we obtain
    \begin{equation*}
        \frac{N}{M} > \frac{2c_pc_1 +3d_1^2}{c_1d_1} = \frac{2c_p}{d_1} + \frac{3d_1}{c_1} \geq 2\sqrt{\frac{6c_p}{c_1}} > \sqrt{\frac{c_p}{c_1}} = \hat{\varepsilon}.
    \end{equation*}
    Thus, the condition
    \begin{align}
        \label{EQUATION_CONDITION_LYAPUNOV_EQUIVALENCE}
        \frac{N}{2} > \frac{M\hat{\varepsilon}}{2} \quad \text{ and } \quad \frac{c_1N}{2} > \frac{Mc_p}{2\hat{\varepsilon}}
    \end{align}
    renders all coefficients in Equation (\ref{STABILITY_LYAPUNOV_ESTIMATION_1})--(\ref{STABILITY_LYAPUNOV_ESTIMATION_2}) strictly positive. Therefore, letting
    \begin{align*}
        k_1 = \frac{\min \left\{N-M\hat{\varepsilon}, c_1N-Mc_p\hat{\varepsilon}^{-1}, \xi_1,\xi_2\right\}}{\max \left\{1, c_1,d_1,d_2\right\}}, \quad
        k_2 = \frac{\max \left\{N+M\hat{\varepsilon}, c_1N+Mc_p\hat{\varepsilon}^{-1}, \xi_1, \xi_2\right\}}{\min \left\{1, c_1,d_1,d_2\right\}},
    \end{align*}
    the equivalence in Equation (\ref{STABILITY_LYAPUNOV_ESTIMATIONS}) holds.

    Combining Equations (\ref{STABILITY_LYAPUNOV_DERIVATIVE_ESTIMATION_2}) and (\ref{STABILITY_LYAPUNOV_ESTIMATIONS}), we arrive at
    $\partial_{t} \mathcal{F}(t) \leq -\frac{\alpha}{k_2} \mathcal{F}(t)$.
    Gronwall's inequality now yields $\mathcal{F}(t) \leq e^{-\frac{\alpha}{k_2}t} \mathcal{F}(0)$.
    Finally, from Equation (\ref{STABILITY_LYAPUNOV_ESTIMATIONS}), we obtain
    \begin{equation*}
        \mathcal{E}(t) \leq C e^{-\frac{2 \alpha}{k_2}t} \mathcal{E}(0) \text{ for } t \geq 0,
    \end{equation*}
    for $C = k_2/k_1$. This finishes the proof.
\end{proof}

\subsection{Exploring the Stability Region in the Parameter Space} 

In Section \ref{SUBSECTION_STABILITY}, we proved that the exponential stability region of Equations (\ref{INTRODUCTION_EQUATION})--(\ref{INTRODUCTION_IC_2}) is non-empty
by showing the natural energy $\mathcal{E}(t)$ from Equation (\ref{STABILITY_ENERGY_DEFINITION}) decays exponentially as $n \to \infty$ 
provided the coefficients $c_1,c_2,d_1,d_2$ satisfy the relations from Assumption \ref{ASSUMPTION_COEFFICIENTS}.
In this Section, we aim to more closely explore the exponential stability region of Equations (\ref{INTRODUCTION_EQUATION})--(\ref{INTRODUCTION_IC_2}) in the parameter space.
A parameter space for (\ref{INTRODUCTION_EQUATION})--(\ref{INTRODUCTION_IC_2}) consists of the tuples
$(\Omega, \Gamma_{0}, \Gamma_{1}, \tau, c_{1}, c_{2}, d_{1}, d_{2})$.
Typically, obtaining necessary and sufficient conditions on the parameters to precisely describe the exponential stability region in the parameter space
reduces to a rigorous analysis of the spectrum of the operator $\mathscr{A}$ on the extended phase space $\mathscr{H}$ (cf. \cite{AmNiPi1995, XuYuLi2006}).
Since the delay differential equation (\ref{INTRODUCTION_EQUATION})--(\ref{INTRODUCTION_IC_2}) is diagonalizable along the eigenbasis of
the negative Laplacian with mixed Dirichlet--Neumann conditions,
the (strong) stability region can also be established by studying the characteristic quasi-polynomial 
of respective `ordinary' delay differential equations for each of the modes (see, e.g., \cite[Chapter 5]{HaVeLu1993}).
Unfortunately, in our situation, both classic approaches are very problematic from the algebraic point of view.

To circumvent this difficulty, instead of uniquely characterizing the parameter space,
while fixing $\Omega, \Gamma_{0}, \Gamma_{1}$ and $\tau$,
we will provide a subset of $c_{1}, c_{2}, d_{1}, d_{2}$, for which the functional in Theorem \ref{THEOREM_LYAPUNOV} is a uniform Lyapunov functional.
Of course, the actual stability region may turn out to be larger than the subset we provide below and remains a task for future investigations.

Recall Equations (\ref{STABILITY_NEGATIVNESS_COEF_1})--(\ref{STABILITY_NEGATIVNESS_COEF_4}) and (\ref{EQUATION_CONDITION_LYAPUNOV_EQUIVALENCE}).
Since all $\varepsilon$'s are supposed to be sufficiently small, letting $\varepsilon_{1} = \varepsilon_{2} = \dots = \min\{\varepsilon_{1}, \varepsilon_{2}, \dots\} =: \varepsilon$,
Equations (\ref{STABILITY_NEGATIVNESS_COEF_1})--(\ref{STABILITY_NEGATIVNESS_COEF_4}) reduce to
\begin{align*}
    -2Mc_1+ M c_2\varepsilon+ M d_1\varepsilon+ M d_2\varepsilon + \xi_1 & < 0, \\
    Nc_2\varepsilon-2Nd_1+Nd_2\varepsilon+ 2Mc_p + \frac{Md_1}{\varepsilon}+\xi_2  & < 0, \\
    \frac{Nc_2}{\varepsilon}+\frac{Mc_2}{\varepsilon}-\xi_1 & < 0, \\
    \frac{Nd_2}{\varepsilon}+\frac{Md_2}{\varepsilon}- \xi_2 & < 0.
\end{align*}
With some elementary linear algebra, the existence of $\xi_1, \xi_2$ is equivalent with
\begin{align*}
    M(2c_1 - c_2\varepsilon - d_1\varepsilon - d_2\varepsilon) & > \frac{Nc_2}{\varepsilon}+\frac{Mc_2}{\varepsilon}, \\
    N(-c_2\varepsilon+2d_1-d_2\varepsilon) - M\Big(2c_p + \frac{d_1}{\varepsilon}\Big) & > \frac{Nd_2}{\varepsilon}+\frac{Md_2}{\varepsilon}.
\end{align*}
Taking into account positivity of $M$, after simple transformations, we get
\begin{align}
    \label{EQUIVALENT_CONDITION_INEQUALITY1}
    \frac{N}{M} \frac{c_2}{\varepsilon} &< 2c_1 - c_2\varepsilon - d_1\varepsilon - d_2\varepsilon - \frac{c_2}{\varepsilon}, \\
    \label{EQUIVALENT_CONDITION_INEQUALITY2}
    \frac{N}{M}\Big(c_2\varepsilon-2d_1+d_2\varepsilon + \frac{d_2}{\varepsilon}\Big) &< -2c_p - \frac{d_1}{\varepsilon}- \frac{d_2}{\varepsilon}.
\end{align}
Equations (\ref{EQUIVALENT_CONDITION_INEQUALITY1})--(\ref{EQUIVALENT_CONDITION_INEQUALITY2}) further yield
\begin{align}
    \label{EQUIVALENT_CONDITION_EQ1}
    c_2\varepsilon-2d_1+d_2\varepsilon + \dfrac{d_2}{\varepsilon} &< 0, \\
    \label{EQUIVALENT_CONDITION_EQ2}
    2c_1 - c_2\varepsilon - d_1\varepsilon - d_2\varepsilon - \frac{c_2}{\varepsilon} &> 0.
\end{align}
Solving for $\varepsilon$, we get
\begin{align}
    \label{EQUIVALENT_CONDITION_FINAL_CONDITIONS1}
    d_1^2 &\geq d_2(c_2+d_2), \quad c_1^2 \geq c_2(c_2+d_1+d_2) \quad \text{ and } \\
    \label{EQUIVALENT_CONDITION_EQ8}
    \varepsilon &\in \Big(\dfrac{d_1-\sqrt{d_1^2-d_2(c_2+d_2)}}{c_2+d_2},\dfrac{d_1+\sqrt{d_1^2-d_2(c_2+d_2)}}{c_2+d_2}\Big), \\
    \label{EQUIVALENT_CONDITION_EQ9}
    \varepsilon &\in \Big(\dfrac{c_1-\sqrt{c_1^2-c_2(c_2+d_1+d_2)}}{c_2+d_1+d_2},\dfrac{c_1+\sqrt{c_1^2-c_2(c_2+d_1+d_2)}}{c_2+d_1+d_2}\Big).
\end{align}
Let us rewrite the last two inequalities in the form
\begin{align}
    \label{EQUIVALENT_CONDITION_EQ3}
    \varepsilon &\in (\underline{\varepsilon},\overline{\varepsilon}), \text{ where } \\
    \notag
    \underline{\varepsilon} & = \max \Big\{\dfrac{d_1-\sqrt{d_1^2-d_2(c_2+d_2)}}{c_2+d_2},\dfrac{c_1-\sqrt{c_1^2-c_2(c_2+d_1+d_2)}}{c_2+d_1+d_2}\Big\}, \\
    \notag
    \overline{\varepsilon} & = \min\Big\{\dfrac{d_1+\sqrt{d_1^2-d_2(c_2+d_2)}}{c_2+d_2},\dfrac{c_1+\sqrt{c_1^2-c_2(c_2+d_1+d_2)}}{c_2+d_1+d_2}\Big\}.
\end{align}

Suppose now $\varepsilon$ satisfies Equation (\ref{EQUIVALENT_CONDITION_EQ3}) and, therefore, Equations (\ref{EQUIVALENT_CONDITION_EQ1})--(\ref{EQUIVALENT_CONDITION_EQ2}).
Then, Equations (\ref{EQUIVALENT_CONDITION_INEQUALITY1})--(\ref{EQUIVALENT_CONDITION_INEQUALITY2}) are equivalent with
\begin{equation}
    \label{EQUIVALENT_CONDITION_EQ4}
    \Big(-2c_p - \frac{d_1}{\varepsilon}- \frac{d_2}{\varepsilon}\Big) \Big(c_2\varepsilon-2d_1+d_2\varepsilon + \frac{d_2}{\varepsilon}\Big)^{-1} < \frac{N}{M} 
    < \Big(2c_1 - c_2\varepsilon - d_1\varepsilon - d_2\varepsilon - \frac{c_2}{\varepsilon}\Big) \Big(\frac{c_2}{\varepsilon}\Big)^{-1}.
\end{equation}
Existence of $\frac{M}{N}$ in the Equation (\ref{EQUIVALENT_CONDITION_EQ4}) is equivalent with
\begin{equation}
    \label{EQUIVALENT_CONDITION_EQ5}
    \Big(-2c_p - \frac{d_1}{\varepsilon}- \frac{d_2}{\varepsilon}\Big) \Big(c_2\varepsilon-2d_1+d_2\varepsilon + \frac{d_2}{\varepsilon}\Big)^{-1} < 
    \Big(2c_1 - c_2\varepsilon - d_1\varepsilon - d_2\varepsilon - \frac{c_2}{\varepsilon}\Big) \Big(\frac{c_2}{\varepsilon}\Big)^{-1}.
\end{equation}
Recalling Equation (\ref{EQUATION_CONDITION_LYAPUNOV_EQUIVALENCE}), we get
\begin{equation}
    \label{EQUIVALENT_CONDITION_EQ6}
    \frac{N}{M} > \sqrt{\frac{c_p}{c_1}}.
\end{equation}
Then, on the strength of Equations (\ref{EQUIVALENT_CONDITION_EQ4}) and (\ref{EQUIVALENT_CONDITION_EQ6}), we have
\begin{equation}
\label{EQUIVALENT_CONDITION_EQ7}
\frac{\varepsilon}{c_2}\left(2c_1 - c_2\varepsilon - d_1\varepsilon - d_2\varepsilon - \frac{c_2}{\varepsilon}\right) > \sqrt{\frac{c_p}{c_1}}.
\end{equation}
Therefore, the existence of $\varepsilon$ is equivalent with fulfilment of inequalities (\ref{EQUIVALENT_CONDITION_EQ3}), (\ref{EQUIVALENT_CONDITION_EQ5}) and (\ref{EQUIVALENT_CONDITION_EQ7}).

\begin{figure}
    \centering
    \includegraphics[scale = 0.5]{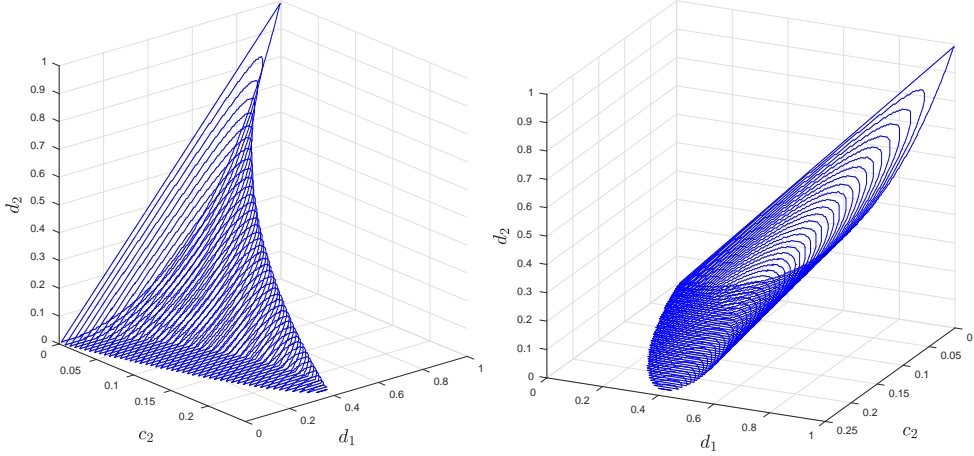}
    \vspace{-0.2in}
    
    \caption{Subset of the stability region generated by Equations (\ref{EQUIVALENT_CONDITION_FINAL_CONDITIONS1}), (\ref{EQUIVALENT_CONDITION_FINAL_CONDITIONS4}) and (\ref{EQUIVALENT_CONDITION_FINAL_CONDITIONS5})
    viewed from two angles \label{FIGURE_STABILITY_REGION}}
\end{figure}

Choose $\varepsilon$ as the center of the interval in Equation (\ref{EQUIVALENT_CONDITION_EQ8})
\begin{equation}
\label{EQUIVALENT_EPSILON_DEFINITION2}
\varepsilon = \frac{d_1}{c_2+d_2}.
\end{equation}
Observing that Equation (\ref{EQUIVALENT_CONDITION_EQ8}) is now trivially satisfied, Equation (\ref{EQUIVALENT_CONDITION_EQ3}) is equivalent with (\ref{EQUIVALENT_CONDITION_EQ9}). 
Further, after simple computations, we infer that Equations (\ref{EQUIVALENT_CONDITION_EQ7}) and (\ref{EQUIVALENT_CONDITION_EQ5}) are equivalent with
\begin{align}
         \label{EQUIVALENT_CONDITION_FINAL_CONDITIONS4}
         \frac{2d_1c_p+(d_1+d_2)(c_2+d_2)}{d_1^2-d_2(c_2+d_2)} < \frac{2c_1d_1}{c_2(c_2+d_2)} - 1 - \frac{d_1^2(c_2+d_1+d_2)}{c_2(c_2+d_2)^2}, \\
         \label{EQUIVALENT_CONDITION_FINAL_CONDITIONS5}
     \sqrt{\frac{c_p}{c_1}} < \frac{d_1}{c_2(c_2+d_2)} \left( 2c_1 - \frac{c_2d_1 + d_1^2 + d_1d_2}{c_2+d_2} - \frac{c_2(c_2+d_2)}{d_1} \right).
\end{align}
respectively. Moreover, we note that (\ref{EQUIVALENT_CONDITION_EQ5}) implies fulfilment of Equation (\ref{EQUIVALENT_CONDITION_EQ2}), and therefore (\ref{EQUIVALENT_CONDITION_EQ9}) holds as well.  Hence, we have proved:
\begin{theorem}
	\label{EQUIVALENT_THEOREM1}
	Let $V^0 \in \mathcal{H}$. Suppose the numbers $c_{1}, c_{2}, d_{1}, d_{2}$ fullfil Equations (\ref{EQUIVALENT_CONDITION_FINAL_CONDITIONS1}), (\ref{EQUIVALENT_CONDITION_FINAL_CONDITIONS4}) and (\ref{EQUIVALENT_CONDITION_FINAL_CONDITIONS5}).
	Then there exist positive constants $\tilde{\alpha}$ and $C$ such that
	\begin{equation*}
	\mathcal{E}(t) \leq Ce^{-2 \tilde{\alpha} t} \mathcal{E}(0) \quad \text{ for } \quad t \geq 0.
	\end{equation*}
\end{theorem}
Letting $c_{1} = 1$ and $c_{p} = 1$,
Figure \ref{FIGURE_STABILITY_REGION} displays the subset of the stability region as defined by Equations
(\ref{EQUIVALENT_CONDITION_FINAL_CONDITIONS1}), (\ref{EQUIVALENT_CONDITION_FINAL_CONDITIONS4}) and (\ref{EQUIVALENT_CONDITION_FINAL_CONDITIONS5}).
Note that fixing $c_{1} = 1$ is no actual restriction since any of the constants $c_{1}, d_{1}, c_{2}, d_{2}$ can be normalized to one
using the substitution $z(t, x) = y(at, x)$ -- without altering the domain and, thus, the Poincar\'{e}'s constant.
From Figure \ref{FIGURE_STABILITY_REGION}, set set appears three-dimensional and non-empty
suggesting the stability region shares the same properties.
While having discovered a non-trivial positive-measure subset of the stability region,
characterizing the latter remains an open problem.


\section{Singular Limits $\tau \to 0$}
\label{SECTION_TAU_TO_ZERO}

In this section, we study the limiting behavior of Equations (\ref{INTRODUCTION_EQUATION})--(\ref{INTRODUCTION_IC_2}) as $\tau \to 0$.
To indicate the dependence on the parameter $\tau$, we rewrite the equations as follows:
\begin{align}
    \notag
    \partial_{tt} y^{\tau}(t, x) - c_1 \triangle y^{\tau}(t, x) - d_1 \triangle \partial_{t} y^{\tau}(t, x) &\phantom{=} \\
    \label{INTRODUCTION_SYSTEM_TAU_EQUATION}
    -c_2 \triangle y^{\tau}(t - \tau, x) - d_2 \partial_t \triangle y^{\tau}(t - \tau, x) &= 0 \text{ for } (t, x) \in (0, \infty) \times \Omega, \\
    \label{INTRODUCTION_SYSTEM_TAU__BC}
    y^{\tau}(t, x) = 0 \text{ for } (t, x) \in (0, \infty) \times \Gamma_0, \quad \frac{\partial y^{\tau}(t, x)}{\partial \nu} &= 0 \text{ for } (t, x) \in (0, \infty) \times \Gamma_1, \\
    \label{INTRODUCTION_SYSTEM_TAU_IC_1}
    y^{\tau}(0+, x) = y^0, \quad \partial_t y^{\tau}(0+, x) &= y^{1} \text{ for } x \in \Omega,\\
    \label{INTRODUCTION_SYSTEM_TAU_IC_2}
    y^{\tau}(t, x) &= \varphi(t, x)  \text{ for } (t, x) \in (-\tau, 0) \times \Omega,
\end{align}
Under appropriate conditions on $y^{0}, y^{1}, \varphi^{0}$, our thrust is to show the sequence of solutions $y^{\tau}$ to Equations
(\ref{INTRODUCTION_SYSTEM_TAU_EQUATION})--(\ref{INTRODUCTION_SYSTEM_TAU_IC_2}) converges in an appropriate topology to the unique solution of the following Kelvin \& Voigt viscoelastic wave equation:
\begin{align}
    \label{EQUATION_LIMITING_SYSTEM_LIMITING_PDE}
    \partial_{tt} y(t, x) - (c_1 + c_2) \triangle y(t,x) - (d_1 + d_2) \triangle \partial_t y(t, x) &= 0 \text{ for } (t, x) \in (0, \infty) \times \Omega, \\
    \label{EQUATION_LIMITING_SYSTEM_LIMITING_BC}
    y(t, x) = 0 \text{ for } x \in \Gamma_0, \quad \frac{\partial y(t, x)}{\partial \nu} &= 0 \text{ for } x \in \Gamma_1, \\
    \label{EQUATION_LIMITING_SYSTEM_LIMITING_IC}
    y(0, x) = y^{0}, \quad \partial_t y(0, x) &= y^{1} \text{ for } x \in \Omega.
\end{align}

\begin{theorem}
    \label{THEOREM_LIMIT_TAU_TO_ZERO}
    
    Let $\tau_{0} > 0$. Suppose $y^{0}, y^{1} \in D(\mathcal{A})$, $\varphi^{0} \in H^{2}\big({-\tau_{0}}, 0; D(\mathcal{A})\big)$ 
    with $\varphi(0+, \cdot) = \varphi^{0}$, $\partial_{t} \varphi(0+, \cdot) = \varphi^{1}$.
    For $0 < \tau \leq \tau_{0}$, let $y^{\tau}$ denote the unique classical solution to Equations (\ref{INTRODUCTION_SYSTEM_TAU_EQUATION})--(\ref{INTRODUCTION_SYSTEM_TAU_IC_2})
    given in Theorem \ref{THEOREM_REGULARITY} and $y$ be the unique classical solution to Equations (\ref{EQUATION_LIMITING_SYSTEM_LIMITING_PDE})--(\ref{EQUATION_LIMITING_SYSTEM_LIMITING_IC})
    provided by Theorem \ref{THEOREM_KALTENBACHER_LASIECKA}.
    Then, for any $T > 0$, there exists a number $C > 0$ such that
    \begin{align*}
        \max_{0 \leq t \leq T} \Big(\big\|y^{\tau}(t, \cdot) &- y(t, \cdot)\big\|_{H^{1}(\Omega)}^{2} + \big\|\partial_{t} y^{\tau}(t, \cdot) - \partial_{t} y(t, \cdot)\big\|_{L^{2}(\Omega)}^{2}\Big) \\
        &\leq C \tau \Big(\|y^{0}\|_{H^{2}(\Omega)}^{2} + \|y^{1}\|_{H^{2}(\Omega)}^{2} + \|\varphi\|_{H^{2}({-\tau_{0}, T}; H^{1}(\Omega))}^{2}\Big)
        \text{ for } 0 < \tau \leq \tau_{0}.
    \end{align*}
\end{theorem}

\begin{proof}
    Introducing the history variable $z^{\tau}(t, s, x) := y^{\tau}(t - s\tau, x)$, we obtain
    \begin{align}
        \label{EQUATION_SYSTEM_TAU_HISTORY_PDE}
        \begin{split}
            \partial_t \nabla z^{\tau}(t, s, x) + \tfrac{1}{\tau} \nabla \partial_s z^{\tau}(t, s, x) &= 0 \text{ for } (t, s, x) \in (0, \infty) \times (0,1) \times \Omega, \\
            \partial_{tt} \nabla z^{\tau}(t, s, x) + \tfrac{1}{\tau} \nabla \partial_{st} z^{\tau}(t, s, x) &= 0 \text{ for } (t, s, x) \in (0, \infty) \times (0,1) \times \Omega
        \end{split}
    \end{align}
    and obtain initial conditions
    \begin{align}
        \label{EQUATION_SYSTEM_TAU_HISTORY_IC}
        z^{\tau}(0, s, x) = \varphi^(-\tau s, x) \quad \text{ and } \quad
        \partial_{t} z^{\tau}(0, s, x) = (\partial_{s} \varphi)(-\tau s, x) \text{ for } (s, x) \in (0, 1) \times \Omega.
    \end{align}
    Since $z$ does not have initial conditions, we artificially impose
    \begin{align}
        \label{EQUATION_LIMITING_SYSTEM_HISTORY_IC}
        z(0, s, x) = \varphi^{0}(-\tau s, x) \quad \text{ and } \quad
        \partial_{t} z(0, s, x) = (\partial_{s} \varphi^{0})(-\tau s, x) \text{ for } (s, x) \in (0, 1) \times \Omega
    \end{align}
    amounting to prescribing initial conditions for $y$ on $[-\tau, 0]$.
    Note that due to the `compatibility condition' $\varphi(0, \cdot) = y^{0}$, $\partial_{t} \varphi(0, \cdot) = y^{1}$,
    the extended $y$ has the same regularity as $y^{\tau}$, in particular, $H^{2}\big({-\tau}, T; D(\mathcal{A}^{1/2}\big)$ (cf. \cite[Lemma 6.1]{KhuPoRa2015}).
    Then, letting $z(t, s, x) := y(t - s\tau, x)$, we get
    \begin{align}
        \label{EQUATION_LIMITING_SYSTEM_HISTORY_PDE}
        \begin{split}
            \partial_t \nabla z(t, s, x) + \tfrac{1}{\tau} \nabla \partial_s z(t, s, x) &= 0 \text{ for } (t, s, x) \in (0, \infty) \times (0,1) \times \Omega, \\
            \partial_{tt} \nabla z(t, s, x) + \tfrac{1}{\tau} \nabla \partial_{st} z(t, s, x) &= 0 \text{ for } (t, s, x) \in (0, \infty) \times (0,1) \times \Omega.
        \end{split}
    \end{align}
    Introducing the solution differences $\bar{y} := y^{\tau} - y$, $\bar{z}^{\tau} := z^{\tau} - z$ and subtracting Equations
    (\ref{INTRODUCTION_SYSTEM_TAU_EQUATION})--(\ref{INTRODUCTION_SYSTEM_TAU_IC_2}), (\ref{EQUATION_SYSTEM_TAU_HISTORY_PDE}), (\ref{EQUATION_SYSTEM_TAU_HISTORY_IC})
    from
    (\ref{EQUATION_LIMITING_SYSTEM_LIMITING_PDE})--(\ref{EQUATION_LIMITING_SYSTEM_LIMITING_IC}), (\ref{EQUATION_LIMITING_SYSTEM_HISTORY_PDE}), (\ref{EQUATION_LIMITING_SYSTEM_HISTORY_IC}),
    respectively, we get
    \begin{align}
        \notag
        \partial_{tt} \bar{y}(t, x) - c_1 \triangle \bar{y}(t,x) - d_1 \triangle \partial_t \bar{y}(t, x)  & \\
        \label{EQUATION_LIMITING_SYSTEM_DIFFERENCE_PDE}
        - c_2 \triangle \bar{z}(t, 1, x) - d_2 \triangle \partial_t \bar{z}(t, 1, x) &= \bar{f}(t, x) \text{ for } (t, x) \in (0, \infty) \times \Omega, \\
        \label{EQUATION_LIMITING_SYSTEM_DIFFERENCE_HISTORY}
        \partial_t \nabla \bar{z}(t, s, x) + \tfrac{1}{\tau} \nabla \partial_s \bar{z}(t, s, x) &= 0 \text{ for } (t, s, x) \in (0, \infty) \times (0,1) \times \Omega, \\
        \label{EQUATION_LIMITING_SYSTEM_DIFFERENCE_HISTORY_DT}
        \partial_{tt} \nabla \bar{z}(t, s, x) + \tfrac{1}{\tau} \nabla \partial_{st} \bar{z}(t, s, x) &= 0 \text{ for } (t, s, x) \in (0, \infty) \times (0,1) \times \Omega, \\
        \label{EQUATION_LIMITING_SYSTEM_DIFFERENCE_BC}
        \bar{z}(t, x) = 0 \text{ for } x \in \Gamma_0, \quad \frac{\partial \bar{z}(t, x)}{\partial \nu} &= 0 \text{ for } x \in \Gamma_1, \\
        \label{EQUATION_LIMITING_SYSTEM_DIFFERENCE_IC_1}
        \bar{y}(0+, x) = 0, \quad \partial_t \bar{y}(0+, x) &= 0 \text{ for } x \in \Omega, \\
        \label{EQUATION_LIMITING_SYSTEM_DIFFERENCE_IC_2}
        \bar{z}(0, s, x) = 0, \quad \partial_t \bar{z}(0, s, x) &= 0 \text{ for } (s, x) \in (0, 1) \times \Omega
    \end{align}
    with
    \begin{equation}
        \bar{f}(t, x) = -c_{2} \big(\triangle y(t, x) - \triangle y(t - \tau, x)\big) -d_{2} \big(\triangle \partial_{t} y(t, x) - \triangle \partial_{t} y(t - \tau, x)\big).
    \end{equation}
    
    Multiplying Equations (\ref{EQUATION_LIMITING_SYSTEM_DIFFERENCE_PDE})--(\ref{EQUATION_LIMITING_SYSTEM_DIFFERENCE_HISTORY_DT}) with $\partial_{t} y(t, \cdot)$ in $L^{2}(\Omega)$
    and $\nabla z(t, \cdot, \cdot)$, $\nabla \partial_{t} z(t, \cdot, \cdot)$ in $L^{2}\big((0, 1) \times \Omega\big)$, respectively, we obtain
    \begin{align*}
        \frac{1}{2} \partial_{t} &\big\|\bar{y}(t, \cdot)\|_{L^{2}(\Omega)}^{2} + \frac{c_{1}}{2} \partial_{t} \big\|\nabla \bar{y}(t, \cdot)\|_{L^{2}(\Omega)}^{2} +
        d_{1} \big\|\nabla \partial_{t} \bar{y}(t, \cdot)\big\|_{L^{2}(\Omega)}^{2} \\
        &\leq
        \varepsilon \big\|\nabla \partial_{t} \bar{y}(t, \cdot)\big\|_{L^{2}(\Omega)}^{2} + \frac{C_{\varepsilon}}{2} \Big(\big\|\nabla \bar{z}(t, 1, \cdot)\big\|_{L^{2}(\Omega)}^{2} +
        \big\|\nabla \partial_{t} \bar{z}(t, 1, \cdot)\big\|_{L^{2}(\Omega)}^{2} + 
        \big\|\bar{f}(t, \cdot)\big\|_{(H^{1}_{\Gamma_{0}}(\Omega))'}^{2}\Big),
    \end{align*}
    where we integrated by parts, used boundary conditions and applied the generalized Young's inequality.
    Selecting $\varepsilon \leq d_{1}$, we obtain
    \begin{align}
        \label{EQUATION_LIMITING_SYSTEM_DIFFERENCE_PDE_ESTIMATED}
        \begin{split}
            \frac{1}{2} \partial_{t} \big\|\bar{y}(t, \cdot)\|_{L^{2}(\Omega)}^{2} &+ \frac{c_{1}}{2} \partial_{t} \big\|\nabla \bar{y}(t, \cdot)\|_{L^{2}(\Omega)}^{2} \\
            &\leq
            \frac{\tilde{C}}{2} \Big(\big\|\nabla \bar{z}(t, 1, \cdot)\big\|_{L^{2}(\Omega)}^{2} +
            \big\|\nabla \partial_{t} \bar{z}(t, 1, \cdot)\big\|_{L^{2}(\Omega)}^{2} + 
            \big\|\bar{f}(t, \cdot)\big\|_{(H^{1}_{\Gamma_{0}}(\Omega))'}^{2}\Big).
        \end{split}
    \end{align}
    for some $\tilde{C} > 0$. Multiplying Equations (\ref{EQUATION_LIMITING_SYSTEM_DIFFERENCE_HISTORY})--(\ref{EQUATION_LIMITING_SYSTEM_DIFFERENCE_HISTORY_DT}) with
    $\tau \tilde{C} \nabla \bar{z}$, $\tau \tilde{C} \nabla \partial_{t} z$ in $L^{2}\big((0, 1) \times \Omega)$, we get
    \begin{align*}
        \frac{\tau \tilde{C}}{2} \partial_{t} \int_{0}^{1} \big\|\nabla z(t, s, \cdot)\big\|_{L^{2}(\Omega)}^{2} \mathrm{d}s +
        \frac{\tilde{C}}{2} \int_{0}^{1} \partial_{s} \big\|\nabla z(t, s, \cdot)\big\|_{L^{2}(\Omega)}^{2} \mathrm{d}s &= 0, \\
        \frac{\tau \tilde{C}}{2} \partial_{t} \int_{0}^{1} \big\|\nabla \partial_{t} z(t, s, \cdot)\big\|_{L^{2}(\Omega)}^{2} \mathrm{d}s +
        \frac{\tilde{C}}{2} \int_{0}^{1} \partial_{s} \big\|\nabla \partial_{t} z(t, s, \cdot)\big\|_{L^{2}(\Omega)}^{2} \mathrm{d}s &= 0
    \end{align*}
    and, therefore,
    \begin{align}
        \label{EQUATION_LIMITING_SYSTEM_DIFFERENCE_HISTORY_TRANSFORMED}
        \begin{split}
            \frac{\tau \tilde{C}}{2} \partial_{t} \int_{0}^{1} \big\|\nabla z(t, s, \cdot)\big\|_{L^{2}(\Omega)}^{2} \mathrm{d}s +
            \frac{\tau \tilde{C}}{2} \Big(\big\|\nabla z(t, 1, \cdot)\big\|_{L^{2}(\Omega)}^{2} - \big\|\nabla z(t, 0, \cdot)\big\|_{L^{2}(\Omega)}^{2}\Big) &= 0, \\
            \frac{\tilde{C}}{2} \partial_{t} \int_{0}^{1} \big\|\nabla \partial_{t} z(t, s, \cdot)\big\|_{L^{2}(\Omega)}^{2} \mathrm{d}s +
            \frac{\tilde{C}}{2} \Big(\big\|\nabla \partial_{t} z(t, 1, \cdot)\big\|_{L^{2}(\Omega)}^{2} - \big\|\nabla \partial_{t} z(t, 0, \cdot)\big\|_{L^{2}(\Omega)}^{2}\Big) &= 0.
        \end{split}
    \end{align}
    Adding up Equations (\ref{EQUATION_LIMITING_SYSTEM_DIFFERENCE_PDE_ESTIMATED}) and recalling
    \begin{equation}
        \notag
        \bar{z}(t, 0, \cdot) = \bar{y}(t, \cdot) \text{ and } \partial_{t} \bar{z}(t, 0, \cdot) = \partial_{t} \bar{y}(t, \cdot),
    \end{equation}
    we obtain
    \begin{align}
        \label{EQUATION_SYSTEM_TAU_DIFFERENCE_ESTIMATE_1}
        \begin{split}
            \partial_{t} \Big(\big\|\bar{y}(t, \cdot)\|_{L^{2}(\Omega)}^{2} &+ c_{1} \big\|\nabla \bar{y}(t, \cdot)\|_{L^{2}(\Omega)}^{2} \\
            &+ \tau \tilde{C} \int_{0}^{1} \big\|\nabla z(t, s, \cdot)\big\|_{L^{2}(\Omega)}^{2} \mathrm{d}s +
            \tau \tilde{C} \int_{0}^{1} \big\|\nabla \partial_{t} z(t, s, \cdot)\big\|_{L^{2}(\Omega)}^{2} \mathrm{d}s\Big) \\
            &\leq \tilde{C} \Big(\big\|\bar{y}(t, \cdot)\|_{L^{2}(\Omega)}^{2} + c_{1} \big\|\nabla \bar{y}(t, \cdot)\|_{L^{2}(\Omega)}^{2}
            + \tilde{C} \big\|\bar{f}(t, \cdot)\big\|_{(H^{1}_{\Gamma_{0}}(\Omega))'}^{2}\Big)
        \end{split}
    \end{align}
    for a suitable $\tilde{C}$ independent of $\tau$.
    Next, we estimate
    \begin{align}
        \label{EQUATION_ESTIMATE_F_BAR}
        \begin{split}
            \max_{0 \leq t \leq T} \big\|\bar{f}(t, &\cdot)\big\|_{(H^{1}_{\Gamma_{0}}(\Omega))'} \leq
            c_{1} \int_{t - \tau}^{t} \big\|\partial_{t} \triangle y(s, \cdot)\big\|_{(H^{1}_{\Gamma_{0}}(\Omega))'} \mathrm{d}s +
            d_{1} \int_{t - \tau}^{t} \big\|\partial_{tt} \triangle y(s, \cdot)\big\|_{(H^{1}_{\Gamma_{0}}(\Omega))'} \mathrm{d}s \\
            &\leq 
            c_{1} \int_{t - \tau}^{t} \big\|\partial_{t} \nabla y(s, \cdot)\big\|_{L^{2}(\Omega)} \mathrm{d}s +
            d_{1} \int_{t - \tau}^{t} \big\|\partial_{tt} \nabla y(s, \cdot)\big\|_{L^{2}(\Omega)} \mathrm{d}s \\
            &\leq 
            c_{1} \sqrt{\tau} \Big(\int_{t - \tau}^{t} \big\|\partial_{t} \nabla y(s, \cdot)\big\|_{L^{2}(\Omega)}^{2} \mathrm{d}s\Big)^{1/2} +
            d_{1} \sqrt{\tau} \Big(\int_{t - \tau}^{t} \big\|\partial_{tt} \nabla y(s, \cdot)\big\|_{L^{2}(\Omega)}^{2} \mathrm{d}s\Big)^{1/2} \\
            &\leq \sqrt{\tau} \max\{c_1, d_1\} \Big(\|y\|_{H^{1}({-\tau_{0}}, T; H^{1}(\Omega))} + \|y\|_{H^{2}({-\tau_{0}}, T; H^{1}(\Omega))}\Big).
        \end{split}
    \end{align}
    Since, by virtue of \cite[Equations (14) and (17)]{KaLa2009}, we have
    \begin{equation}
        \notag
        \begin{split}
            \int_{-\tau}^{T} \Big(\big\|\partial_{t} y(t, \cdot)\big\|_{H^{1}(\Omega)}^{2} &+ \big\|\partial_{tt} y(t, \cdot)\big\|_{H^{1}(\Omega)}^{2}\Big) \mathrm{d}s \\
            &\leq \frac{M \big(e^{\omega T} - 1\big)}{\omega} \Big(\|y^{0}\|_{H^{2}(\Omega)}^{2} + \|y^{1}\|_{H^{2}(\Omega)}^{2}\Big) + \|\varphi\|_{H^{2}({-\tau_{0}}, T; H^{1}(\Omega))}^{2},
        \end{split}
    \end{equation}
    for some $M, \omega > 0$, Equation (\ref{EQUATION_ESTIMATE_F_BAR}) yields
    \begin{equation}
        \label{EQUATION_ESTIMATE_F_BAR_FINITE}
        \max_{0 \leq t \leq T} \big\|\bar{f}(t, \cdot)\big\|_{(H^{1}_{\Gamma_{0}}(\Omega))'}^{2} \leq
        \check{C} \tau \Big(\|y^{0}\|_{H^{2}(\Omega)}^{2} + \|y^{1}\|_{H^{2}(\Omega)}^{2} + \|\varphi\|_{H^{2}({-\tau_{0}}, T; H^{1}(\Omega))}^{2}\Big) 
    \end{equation}
    for some $\check{C} = \check{C}(T) > 0$.
    Trivially adding $\tau \tilde{C} \int_{0}^{1} \big\|\nabla z(t, s, \cdot)\big\|_{L^{2}(\Omega)}^{2} \mathrm{d}s + 
    \tau \tilde{C} \int_{0}^{1} \big\|\nabla \partial_{t} z(t, s, \cdot)\big\|_{L^{2}(\Omega)}^{2} \mathrm{d}s\Big)$
    to the right-hand side of Equation (\ref{EQUATION_SYSTEM_TAU_DIFFERENCE_ESTIMATE_1}), combining the resulting estimate with Equation (\ref{EQUATION_ESTIMATE_F_BAR_FINITE})
    and applying Gronwall's inequality, we arrive at
    \begin{equation}
        \notag
        \mathcal{E}_{\tau}(t) \leq \bar{C} \Big(\|y^{0}\|_{H^{2}(\Omega)}^{2} + \|y^{1}\|_{H^{2}(\Omega)}^{2} + \|\varphi\|_{H^{2}({-\tau_{0}, 0}; H^{1}(\Omega))}^{2}\Big)
        \text{ for } 0 \leq t \leq T.
    \end{equation}
    Since the latter estimate holds iniformly in $\tau$, the claim follows.
\end{proof}

\begin{remark}
    Under conditions of Theorems \ref{THEOREM_LYAPUNOV} or \ref{EQUIVALENT_THEOREM1},
    due to exponential stability of both systems (\ref{EQUATION_LIMITING_SYSTEM_LIMITING_PDE})--(\ref{EQUATION_LIMITING_SYSTEM_LIMITING_IC}) (cf. \cite{KaLa2009})
    and (\ref{INTRODUCTION_SYSTEM_TAU_EQUATION})--(\ref{INTRODUCTION_SYSTEM_TAU_IC_2}),
    the proof of Theorem \ref{THEOREM_LIMIT_TAU_TO_ZERO} can easily be amended to obtain a global-in-time convergence:
    \begin{align*}
        \sup_{0 \leq t \leq \infty} \Big(\big\|y^{\tau}(t, \cdot) &- y(t, \cdot)\big\|_{H^{1}(\Omega)}^{2} + \big\|\partial_{t} y^{\tau}(t, \cdot) - \partial_{t} y(t, \cdot)\big\|_{L^{2}(\Omega)}^{2}\Big) \\
        &\leq C \tau \Big(\|y^{0}\|_{H^{2}(\Omega)}^{2} + \|y^{1}\|_{H^{2}(\Omega)}^{2} + \|\varphi\|_{H^{2}({-\tau_{0}, T}; H^{1}(\Omega))}^{2}\Big)
        \text{ for } 0 < \tau \leq \tau_{0}.
    \end{align*}
\end{remark}

\section{Numerical Example} \label{SECTION_NUMERICAL_EXAMPLE}
In this final section, we present a numerical example to illustrate our model and compare it with its classic instanteneous version.
Motivated by \cite{MoLe2007}, we consider a viscoelatic rod-shaped muscle sample of length $L > 0$ (with a sufficiently small cross-section)
fixed at its left boundary $x = 0$, undeformed and free of stresses for the time $t \leq 0$ 
and subject to a constant load $f$ at $x = L$ pointing into the direction of the outer normal for all times $t > 0$.
Let $\rho, E$ and $\eta$ denote the (uniform) mass density, Young's modulus and viscosity, respectively.
For Kelvin \& Voigt viscoelastic materials, one commonly considers a characteristic quantity $\tau_{r} = \eta/E$ referred to as `retardation time'.
For illustration purposes, we let the delay time $\tau$ be equal $\tau_{r}$ and define
\begin{equation}
    \label{EQUATION_EXAMPLE_CONSTANTS_C_D}
    c_{1} := E/\rho, \quad d_{1} := \eta/\rho, \quad c_{2} := \varepsilon E/\rho \quad \text{ and } \quad d_{2} := \varepsilon \eta/\rho \notag
\end{equation}
for a small positive $\varepsilon > 0$ to be chosen later. Table \ref{TABLE_PHYSICAL_CONSTANTS} summarizes the physical constants introduced above.
\begin{table}[!h]
    \centering
    
    \begin{tabular}{cccc}
        Notation & Units        & Value                  & Description \\
        \hline \hline
        $L$        & m            & $5.33 \cdot 10^{-3}$   & rod length                            \\
        $\rho$     & kg/m$^{3}$   & $1.06 \cdot 10^{3}$    & mass density                          \\
        $E$        & Pa           & $2.00 \cdot 10^{4}$    & Young's modulus                       \\
        $\eta$     & Pa $\cdot$ s & $2.00 \cdot 10^{7}$    & viscosity                             \\
        $\tau$     & s            & $1.00 \cdot 10^{3}$    & delay/retardation time                \\
        $f$        & Pa           & $1.0052 \cdot 10^{4}$  & (outward) surface traction at $x = L$ \\
        \hline
    \end{tabular}
    
    \caption{Physical constants of the model (cf. \cite[Table 3]{MoLe2007}) \label{TABLE_PHYSICAL_CONSTANTS}}
\end{table}

Letting $y(t, x)$ denote the longitudinal displacement of the material point $x \in (0, L)$ in the reference configuration at time $t$, 
with the notation from Equation (\ref{EQUATION_EXAMPLE_CONSTANTS_C_D}),
the one-dimensional version of Equations (\ref{INTRODUCTION_EQUATION})--(\ref{INTRODUCTION_IC_2}) describing the deformation of the sample reads as
\begin{align}
    \notag
    \partial_{tt} y(t, x) - c_1 \partial_{xx} y(t, x) - c_2 \partial_{xx} y(t-\tau, x) &\phantom{=} \\
    \label{EXAMPLE_EQUATION_PDE}
    -d_1 \partial_t \partial_{xx} y(t,x) - d_2 \partial_t \partial_{xx} y(t - \tau, x) &= 0 \text{ for } t > 0, \quad x \in (0, L) \\
    \label{EXAMPLE_EQUATION_BC_1}
    y(t, 0) = 0, \quad c_1 \partial_{x} y(t, L) + c_2 \partial_{x} y(t - \tau, L) &\phantom{=} \\
    \label{EXAMPLE_EQUATION_BC_2}
    + d_1 \partial_t \partial_{x} y(t, L) + d_2 \partial_t \partial_{x} y(t - \tau, L) &= f \text{ for } t > 0, \\
    \label{EXAMPLE_EQUATION_IC}
    y(0+, x) = 0, \quad \partial_{t} y(0+, x), \quad
    y(t, x) &= 0 \text{ for } t \in (-\tau, 0), \quad x \in (0, L).
\end{align}
To make the boundary condition (\ref{EXAMPLE_EQUATION_BC_2}) compatible with Equation (\ref{INTRODUCTION_BC_2}) and the theory developed in Section \ref{SECTION_RADON_MEASURE_SOLUTION},
it needs to be reformulated in terms of $\psi(t) := \frac{\partial y(t, L)}{\partial \nu}$:
\begin{equation}
    \label{EQUATION_NEUTRAL_DDE}
    d_{1} \dot{\psi}(t) + c_{1} \psi(t) + c_{2} \psi(t - \tau) + d_{2} \dot{\psi}(t - \tau) = f \text{ for } t > 0, \quad
    \psi(t) = 0 \text{ for } t \in [-\tau, 0].
\end{equation}
For $\varepsilon = 0$ (viz. $c_{2} = d_{2} = 0$), Equation (\ref{EQUATION_NEUTRAL_DDE}) reduces to an ODE uniquely solvable by
$\psi(t) = \frac{f}{E}\Big(1 - e^{-\frac{Et}{\eta}}\Big) \mathds{1}_{t \geq 0}$.
For $\varepsilon > 0$, Equation (\ref{EQUATION_NEUTRAL_DDE}) is a neutral differential equation
and is known to possess a global unique solution $\psi \in C^{\infty}\big([-\tau, 0]\big) \cap C^{1}\big([0, \infty)\big)$ (cf. \cite{HaVeLu1993, KoNiPhi1998})
plotted in Figure \ref{FIGURE_SOLUTION_PSI}.
\begin{figure}
    \centering
    \includegraphics[scale = 0.4]{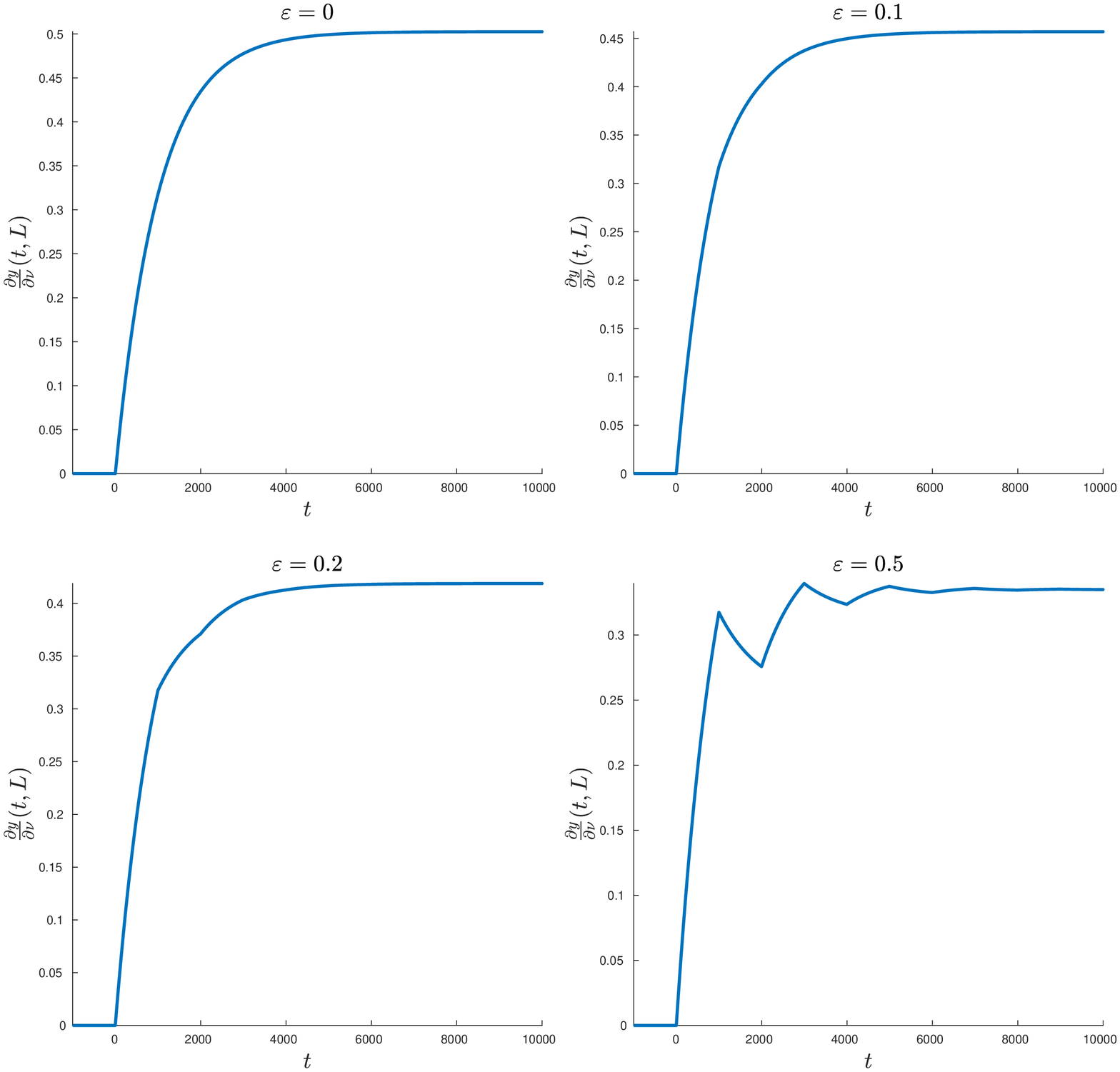}
    \vspace{-0.2in}
    
    \caption{Numerical solution to Equation (\ref{EQUATION_NEUTRAL_DDE}) for $\varepsilon = 0.0, 0.1, 0.2, 0.5$ \label{FIGURE_SOLUTION_PSI}}
\end{figure}

Introducing 
\begin{equation}
    \label{EQUATION_EXAMPLE_SUBSTITUTION_FOR_Y}
    w(t, x) := y(t, x) - \Big(\psi(t) - \frac{\dot{\psi}(0+)}{2} |t|\Big) x
\end{equation}
and observing that $\ddot{\psi}$ is a real-valued Radon measure on $[0, \infty)$ with $\ddot{\psi}\big(\{0\}\big) = \dot{\psi}(0+) \delta_{0}$, 
where $\delta_{0}(\cdot)$ is the Dirac distribution supported at the origin, we obtain
\begin{align}
    \notag
    \partial_{tt} w(t, x) - c_1 \partial_{xx} w(t, x) - c_2 \partial_{xx} w(t-\tau, x) &\phantom{=} \\
    \label{EXAMPLE_EQUATION_TRANSFORMED_PDE}
    -d_1 \partial_t \partial_{xx} w(t,x) - d_2 \partial_t \partial_{xx} w(t - \tau, x) &= \Big(\ddot{\psi}(t) - \delta_{0}(t) \dot{\psi}(0+)\Big) x \text{ for } t > 0, \quad x \in (0, L), \\
    \label{EXAMPLE_EQUATION_TRANSFORMED_BC_1}
    w(t, 0) = 0, \quad \frac{\partial w}{\partial \nu}(t, L) &= 0 \text{ for } t > 0, \\
    \label{EXAMPLE_EQUATION_TRANSFORMED_IC}
    w(0+, x) = 0, \quad \partial_{t} w(0+, 0) = 0, \quad w(t, x) &= 0 \text{ for } t \in [-\tau, 0), \quad x \in [0, L].
\end{align}
Since the right-rand side of Equation (\ref{EXAMPLE_EQUATION_TRANSFORMED_PDE}) is a Radon measure regular at $0$,
on the strength of Corollary \ref{COROLLARY_CAUCHY_PROBLEM_RADON_MEASURE_EXISTENCE},
Equations (\ref{EXAMPLE_EQUATION_TRANSFORMED_PDE})--(\ref{EXAMPLE_EQUATION_TRANSFORMED_IC}) possess a unique extrapolated BV-solution
$w$ with $w \in H^{1}\big({-\tau}, 0; \mathcal{H}\big) \cap \operatorname{BV}_{\mathrm{loc}}\big([0, \infty), D(\mathcal{A}^{1/2})\big)$,
$\partial_{t} w \in L^{2}\big({-\tau}, 0; \mathcal{H}\big) \cap \operatorname{BV}_{\mathrm{loc}}\big([0, \infty), \mathcal{H}\big)$, 
where the space $\mathcal{H} := L^{2}\big((0, L)\big)$ is equipped with the usual inner product and $\mathcal{A} \equiv -\partial_{xx} \colon D(\mathcal{A}) \to \mathcal{H}$ with
\begin{equation}
    D(\mathcal{A}) = \big\{w \in H^{2}\big((0, L)\big) \,|\, w(0) = 0, \quad \partial_{x} w(L) = 0\big\}. \notag
\end{equation}

\begin{figure}
    \centering
    \includegraphics[scale = 0.6]{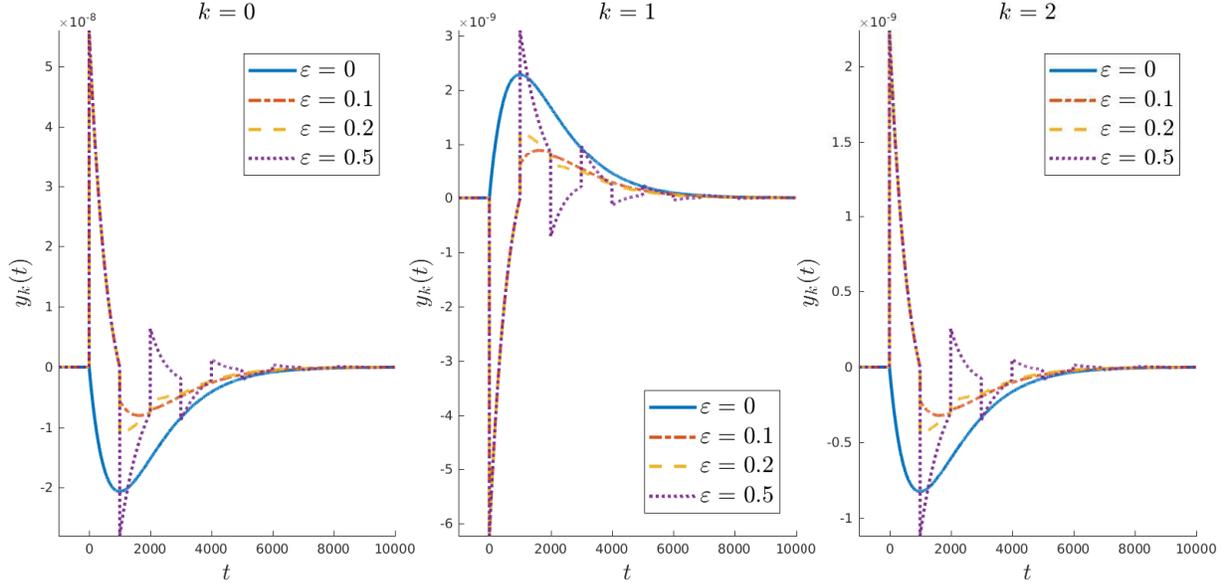}
    
    \caption{Numerical approximation of the modes $w_{n}$ from Equation (\ref{EQUATION_EXAMPLE_MODES}) \label{FIGURE_MODES}}
\end{figure}

The operator $\mathcal{A}$ is known to be diagonalizable via
\begin{align}
    \mathcal{A} &= \sum_{n = 0}^{\infty} \lambda_{n} \langle \cdot, \phi_{n}\rangle_{\mathcal{H}} \phi_{n} \quad \text{ with } \\
    \lambda_{n} &= \frac{\pi^{2} (2k + 1)^{2}}{4 L^{2}} \quad \text{ and } \quad
    \phi_{n}(x) = \sqrt{\frac{2}{L}} \sin\Big(\frac{2 \pi (k + 1) x}{2L}\Big) \text{ for } n \in \mathbb{N}_{0}.
\end{align}
Further, we can write
\begin{equation}
    F(t, \cdot) = \sum_{n = 0}^{\infty} F_{n}(t) \phi_{n}(x) \text{ with }
    F_{n}(t) = \sqrt{\frac{2}{L}} \frac{4 L^{2}}{\pi^{2}} \frac{(-1)^{k}}{(2k + 1)^{2}} \Big(\psi(t) - \frac{\dot{\psi}(0+)}{2} |t|\Big) \text{ for } n \in \mathbb{N}_{0}.
\end{equation}
Thus, letting
\begin{equation}
    \label{EQUATION_EXAMPLE_ANSATZ_W}
    w(t, \cdot) = \sum_{n = 0}^{\infty} w_{n}(t) \phi_{n}, 
\end{equation}
Equations (\ref{EXAMPLE_EQUATION_TRANSFORMED_PDE})--(\ref{EXAMPLE_EQUATION_TRANSFORMED_IC}) can be reduced to solving the sequence of decoupled `ordinary' delay differential equations
with an impulse right-hand side
\begin{equation}
    \label{EQUATION_EXAMPLE_MODES}
    \begin{split}
        \ddot{w}_{n}(t) + c_{1} \lambda_{n} w_{n}(t) + c_{2} \lambda_{n} \dot{w}_{n}(t)
        + c_{2} \lambda_{n} w_{n}(t - \tau) + c_{2} \lambda_{n} \dot{w}_{n}(t - \tau) &= \ddot{F}_{n}(t) \text{ for } t > 0, \\
        w_{n}(0+) = \dot{w}_{n}(t) = 0, \quad w_{n}(t) &= 0 \text{ for } t \in [-\tau, 0)
    \end{split}
\end{equation}
for $n_{0} \in \mathbb{N}_{0}$.
By Theorem \ref{THEOREM_CAUCHY_PROBLEM_RADON_MEASURE_EXISTENCE}, 
for each $n \in \mathbb{N}$, there exists a unique solution $w_{n}$ with $w_{n}, \dot{w}_{n} \in \operatorname{BV}_{\mathrm{loc}}\big([0, \infty)\big)$.

Figure \ref{FIGURE_MODES} displays the numerical solution to Equations (\ref{EQUATION_EXAMPLE_MODES}) obtained with the standard Crank \& Nicholson time integrator
with the same integrator employed to solve the neutral differential equation (\ref{EQUATION_NEUTRAL_DDE}).
The second-order derivative $\ddot{\psi}(t)$ was approximated with a 3-point backward differential quotient formula.
We truncated the time interval at $T = 10 \tau$ and chose a time step $h \approx 10^{-3} \tau$.
\ifdefined\MATLAB
A \texttt{Matlab}$^{\text{\textregistered}}$ implementation of the solver can be found in the Appendix.
\fi
\begin{figure}
    \centering
    \includegraphics[scale = 0.6]{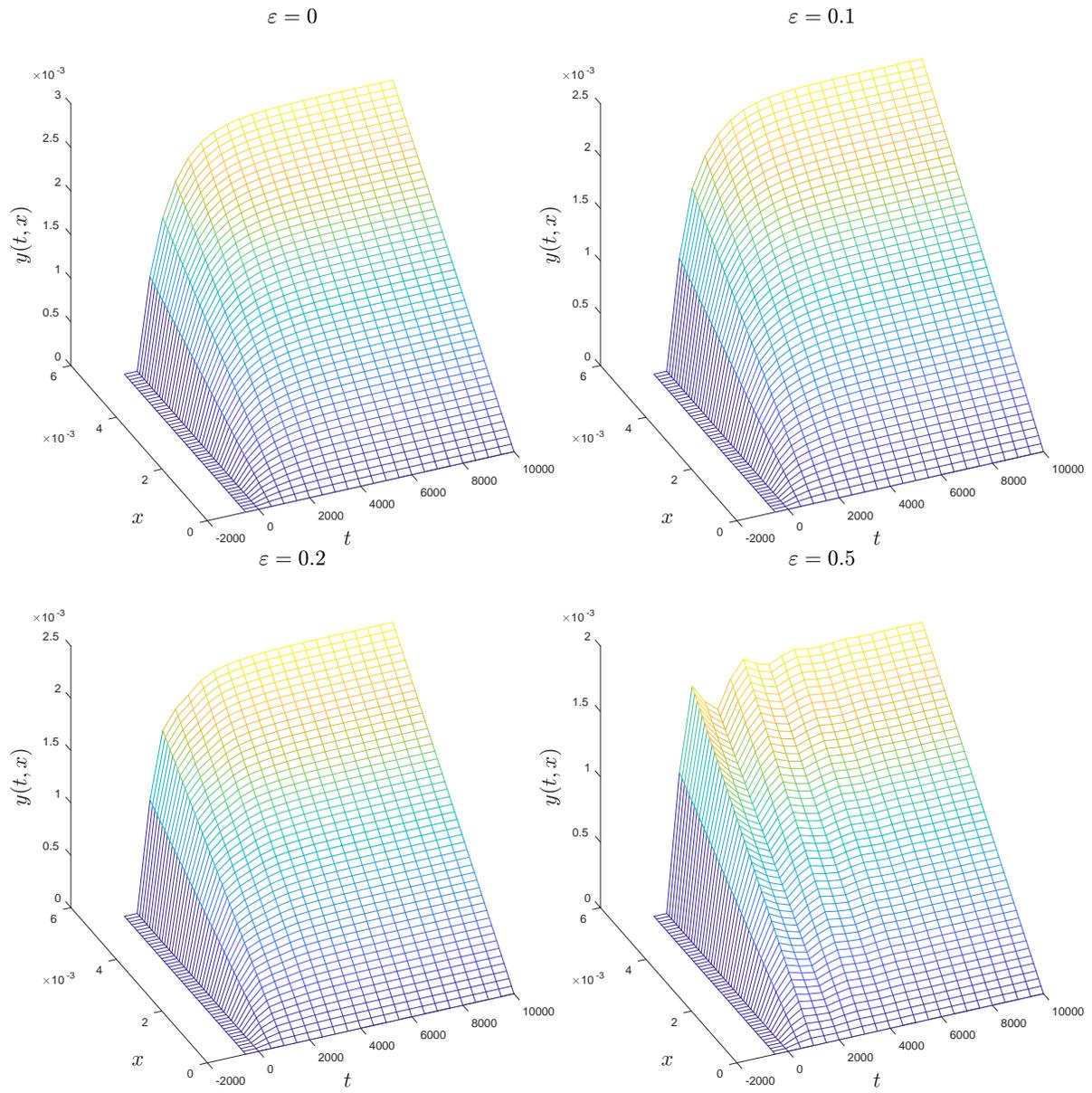}
    
    \caption{Numerical solution to Equations (\ref{EXAMPLE_EQUATION_TRANSFORMED_PDE})--(\ref{EXAMPLE_EQUATION_TRANSFORMED_IC}) for $\varepsilon = 0.0, 0.1, 0.2, 0.5$ \label{FIGURE_SOLUTION_Y}}
\end{figure}

Resubstituting the numerical solutions $w_{n}(t)$ into Equation (\ref{EQUATION_EXAMPLE_ANSATZ_W})
and plugging the result along with the numerical solution $\psi(t)$ into Equation (\ref{EQUATION_EXAMPLE_SUBSTITUTION_FOR_Y}),
we obtain a numerical solution $y(t, x)$ for Equations (\ref{EXAMPLE_EQUATION_PDE})--(\ref{EXAMPLE_EQUATION_IC}).
The latter is displayed in Figure \ref{FIGURE_SOLUTION_Y} for $\varepsilon = 0.0, 0.1, 0.2, 0.5$.
The solutions appear to be quite consistent with each other among all $\varepsilon$'s considered
and almost perfectly agree with the 2D results reported by Moravec \& Letzelter \cite[Figure 3, KV model]{MoLe2007}, especially for $\varepsilon < 0.5$.
Our observations go in line with our expectations outlined in the Introduction Section \ref{SECTION_INTRODUCTION}
stating that the delay material law from Equation (\ref{EQUATION_HOOK_LAW_KELVIN_VOIGT_DELAY})
would allow for a more complicated, but stable dynamics catching non-monotonic patters
as opposed to the classic instanteneous constitutional relation (\ref{EQUATION_HOOK_LAW_KELVIN_VOIGT}).


\section{Summary and Conclusions}
We proposed a model of a viscoelastic Kelvin \& Voigt material with a time-localized delay material law
and established a well-posedness theory for the underlying initial-value problem.
Further, we investigated the long-time behavior of the system including a a discussion on the stability region in the parameter space.
Finally, we gave a numerical study of our model applied to a real-world example.
Our future investigations will include nonlinear material laws incorporating both mixed time-localized and integral time delay terms.

\section*{Acknowledgements}
The work on this project was supported by the German Research Foundation (DFG) through the Collaborative Research Center 1173 ``Wave Phenomena,''
a start-up grant from the University of Texas at El Paso and the Erasmus+ Mobility Program (participation contract No. 128/2015-2016).


\bibliographystyle{plain}
\bibliography{bibliography} 

\ifdefined\MATLAB
\appendix

\section*{Appendix: \texttt{Matlab}$^{\text{\textregistered}}$ Code}
{\footnotesize
\begin{verbatim}
%%%%%%%%%%%%%%%%%%%%%%%%%%%%%%%%%%%%%%%%%%%%%%%%%%%%%%%%%%%%%%%%%%%%%%%%%%%
%                    Kelvin & Voigt viscoelastic rod                      %
%                                                                         %
% (C) Michael Pokojovy 2018                                               %
%%%%%%%%%%%%%%%%%%%%%%%%%%%%%%%%%%%%%%%%%%%%%%%%%%%%%%%%%%%%%%%%%%%%%%%%%%%

function example
    %%%%%%%%%%%%%%%%%%%%%%%%%%%%%%%%%%%%%%%%%%%%%%%%%%%%%%%%%%%%%%%%%%%%%%%
    % Creep problem for a Kelvin-Voigt viscoelastic muscle sample         %
    %%%%%%%%%%%%%%%%%%%%%%%%%%%%%%%%%%%%%%%%%%%%%%%%%%%%%%%%%%%%%%%%%%%%%%%

    global L rho E eta epsilon tau f;
    
    L = 5.33E-3;      % rod length [m]
    
    rho = 1.06E3;     % mass density  [kg/m^3]
    E   = 2.00E4;     % Young modulus [Pa]
    eta = 2.00E7;     % viscosity     [Pa x s]
    
    f = 1.0052E4;     % surface traction applied at x = L [Pa]
    
    tau = eta/E;      % 'retardation time' [s]
    
    N = 1000;         % number of grid points per delay period
    K = 10;           % number of delay periods past 0
    
    %%%%%%%%%%%%%%%%%%%%%%%%%%%%%%%%%%%%%%%%%%%%%%%%%%%%%%%%%%%%%%%%%%%%%%%
    figure(1);
    
    set(gcf, 'PaperUnits', 'centimeters');
    xSize = 28; ySize = 28;
    xLeft = (21 - xSize)/2; yTop = (30 - ySize)/2;
    set(gcf,'PaperPosition', [xLeft yTop xSize ySize]);
    set(gcf,'Position', [0 0 xSize*50 ySize*50]);
    
    epsilons = [0 0.1 0.2 0.5];
    
    for k = 1:length(epsilons)
        subplot_tight(2, 2, k, [0.09 0.075]);
        hold on;
        
        epsilon = epsilons(k);
        
        title(['$\varepsilon = ', num2str(epsilon), '$'], ...
               'interpreter', 'latex', 'FontSize', 20);
        
        T = linspace(-tau, K*tau, (K + 1)*N);
        Y_bc = BC(K, N);
        
        plot(T, Y_bc , 'LineWidth', 2);
        
        axis([min(T) max(T) min(Y_bc) max(Y_bc)]);
        
        xlabel('$t$', 'interpreter', 'latex', 'FontSize', 20);
        ylabel('$\frac{\partial y}{\partial \nu}(t, L)$', ...
               'interpreter', 'latex', 'FontSize', 20);
    end 
    
    %%%%%%%%%%%%%%%%%%%%%%%%%%%%%%%%%%%%%%%%%%%%%%%%%%%%%%%%%%%%%%%%%%%%%%%
    figure(2);
    
    set(gcf, 'PaperUnits', 'centimeters');
    xSize = 28; ySize = 14;
    xLeft = (21 - xSize)/2; yTop = (30 - ySize)/2;
    set(gcf, 'PaperPosition', [xLeft yTop xSize ySize]);
    set(gcf, 'Position', [0 0 xSize*50 ySize*50]);
    
    ks = [0 1 2];
    
    for i = 1:3
        k = ks(i);
        
        epsilon = 0;
        [T, Y_nought]  = ODEmode(k, K, N);
        
        mx = min(T);
        Mx = max(T);
        my = min(Y_nought(1, :));
        My = max(Y_nought(1, :));
 
        epsilons = [0.1 0.2 0.5];
                
        subplot_tight(1, 3, i, [0.1 0.05]);
        title(['$k = ', num2str(k), '$'], ...
               'interpreter', 'latex', 'FontSize', 20);
        hold on;
        plot(T, Y_nought(1, :), '-', 'LineWidth', 2);
        
        line_types = {'-.', '--', ':'};
        
        for j = 1:length(epsilons)
            epsilon = epsilons(j);

            [T, Y_tau] = DDEmode(k, K, N);

            my = min(my, min(Y_tau(1, :)));
            My = max(My, max(Y_tau(1, :)));
            
            plot(T, Y_tau(1, :), line_types{j}, 'LineWidth', 2); 
        end
        
        legend({['$\varepsilon = 0$'], ...
                ['$\varepsilon = ', num2str(epsilons(1)), '$'], ...
                ['$\varepsilon = ', num2str(epsilons(2)), '$'], ...
                ['$\varepsilon = ', num2str(epsilons(3)), '$']}, ...
                 'Location', 'Best', 'interpreter', 'latex', 'FontSize', 20);
        
        axis([mx Mx my My]);  
        
        xlabel('$t$', 'interpreter', 'latex', 'FontSize', 20);
        ylabel('$y_{k}(t)$', 'interpreter', 'latex', 'FontSize', 20);
    end
    
    %%%%%%%%%%%%%%%%%%%%%%%%%%%%%%%%%%%%%%%%%%%%%%%%%%%%%%%%%%%%%%%%%%%%%%%
    figure(3);
    
    set(gcf, 'PaperUnits', 'centimeters');
    xSize = 28; ySize = 28;
    xLeft = (21 - xSize)/2; yTop = (30 - ySize)/2;
    set(gcf,'PaperPosition', [xLeft yTop xSize ySize]);
    set(gcf,'Position', [0 0 xSize*50 ySize*50]);
    
    %%%
    subplot_tight(2, 2, 1, [0.05 0.075]);
    hold on;
    title(['$\varepsilon = ', num2str(0), '$'], ...
           'interpreter', 'latex', 'FontSize', 20);
    
    epsilon = 0;
    
    Ind = (1:((K + 1)*N/100))*100;
    T_sel = T(Ind);
    X_sel = linspace(0, L, 50);
    
    [T_grid, X_grid] = meshgrid(T_sel, X_sel);
    
    Y_grid = zeros(size(X_grid));
    Y_bc = BC(K, N);
    
    for k = 0:20
        [~, Y_tau] = ODEmode(k, K, N);
        
        for i = 1:size(T_grid, 1)
            Y_grid(i, :) = Y_grid(i, :) + ...
            (Y_tau(1, Ind)).*(sqrt(2/L)*sin(pi/L*(k + 0.5)*X_grid(i, :)));
        end
    end

    for j = 1:size(T_grid, 1)
        Y_grid(j, :) = Y_grid(j, :) + Y_bc(Ind).*X_grid(j, :);
    end
    
    xlabel('$t$', 'interpreter', 'latex', 'FontSize', 20);
    ylabel('$x$', 'interpreter', 'latex', 'FontSize', 20);
    zlabel('$y(t, x)$', 'interpreter', 'latex', 'FontSize', 20);
    
    mesh(T_grid, X_grid, Y_grid);
    view([-24 33]);
    
    %%%
    for i = 1:3
        subplot_tight(2, 2, i + 1, [0.05 0.075]);
        hold on;
        title(['$\varepsilon = ', num2str(epsilons(i)), '$'], ...
               'interpreter', 'latex', 'FontSize', 20);
        
        epsilon = epsilons(i);

        Ind = (1:((K + 1)*N/100))*100;
        T_sel = T(Ind);
        X_sel = linspace(0, L, 50);

        [T_grid, X_grid] = meshgrid(T_sel, X_sel);

        Y_grid = zeros(size(X_grid));
        Y_bc = BC(K, N);

        for k = 0:20
            [~, Y_tau] = DDEmode(k, K, N);

            for j = 1:size(T_grid, 1)
                Y_grid(j, :) = Y_grid(j, :) + ...
                (Y_tau(1, Ind)).*(sqrt(2/L)*sin(pi/L*(k + 0.5)*X_grid(j, :)));
            end
        end
        
        for j = 1:size(T_grid, 1)
            Y_grid(j, :) = Y_grid(j, :) + Y_bc(Ind).*X_grid(j, :);
        end
        
        xlabel('$t$', 'interpreter', 'latex', 'FontSize', 20);
        ylabel('$x$', 'interpreter', 'latex', 'FontSize', 20);
        zlabel('$y(t, x)$', 'interpreter', 'latex', 'FontSize', 20);

        mesh(T_grid, X_grid, Y_grid);
        view([-24 33]);
    end
end

function [T, Y] = ODEmode(k, K, N)
    % Solves the ODE for the k-th mode

    global L rho E eta epsilon tau f;
    
    c1 = (1 + epsilon)*E/rho;
    d1 = (1 + epsilon)*eta/rho;
    
    lambda = (pi/L*(k + 0.5))^2;
    A = [ 0          1;
         -c1*lambda -d1*lambda];
         
    T = linspace(-tau, K*tau, (K + 1)*N);
    Y = zeros(2, length(T));
    
    for j = 1:N
        Y(:, j) = 0;
    end
    
    Y_bc = BC(K, N);
    
    for j = (N + 1):length(T)
        dt = T(j) - T(j - 1);

        F = 4*sqrt(2/L)*L^2*(-1)^k/(pi^2*(2*k + 1)^2)*[1; 0]* ...
            (-f*E/eta^2*exp(-E*T(j)/eta));
        
        Y(:, j) = (eye(2) - 0.5*dt*A)\((eye(2) + 0.5*dt*A)*Y(:, j - 1) + 0.5*dt*F);
    end
end

function [T, Y] = DDEmode(k, K, N)
    % Solves the DDE for the k-th mode

    global L rho E eta epsilon tau;

    c1 = E/rho;
    d1 = eta/rho;
    
    c2 = c1*epsilon;
    d2 = d1*epsilon;
    
    lambda = (pi/L*(k + 0.5))^2;
    A = [ 0              1;
         -c1*lambda -d1*lambda];
    
    B = [ 0              0;
         -c2*lambda -d2*lambda];
         
    T = linspace(-tau, K*tau, (K + 1)*N);
    Y = zeros(2, length(T));
    
    Y_bc = BC(K, N);
    
    for j = (N + 1):length(T)
        dt = T(j) - T(j - 1);
        
        F = 4*sqrt(2/L)*L^2*(-1)^k/(pi^2*(2*k + 1)^2)*[1; 0]* ...
            (Y_bc(j) - 2*Y_bc(j - 1) + Y_bc(j - 2))/dt^2; 
        
        Y(:, j) = (eye(2) - 0.5*dt*A)\((eye(2) + 0.5*dt*A)*Y(:, j - 1) + ...
                   0.5*dt*B*(Y(:, j - N + 1) + Y(:, j - N)) + 0.5*dt*F);
    end
end

function Y = BC(K, N)
    % Solve initial condition for $\partial y/\partial \nu$

    global E eta epsilon tau f;

    T = linspace(-tau, K*tau, (K + 1)*N);
    Y = zeros(1, length(T));
    
    for j = (N + 1):length(T)
        dt = T(j) - T(j - 1);
        
        a = -E/eta;
        b = -epsilon*E/eta;
        c = -epsilon;
        
        if (epsilon == 0)
            Y(j) = f/E*(1 - exp(-E*T(j)/eta));
        else    
            Y(j) = (1 - 0.5*dt*a)\(dt*f/eta + (1 + 0.5*a*dt)*Y(j - 1) + ...
                    0.5*b*dt*Y(j - N) + 0.5*b*dt*Y(j - N + 1) + ...
                    0.5*c*(Y(j - N + 2) - Y(j - N)));
        end
    end
end
\end{verbatim}
}
\fi

\end{document}